\documentclass{article}

\usepackage[linesnumbered,ruled,vlined]{algorithm2e}  

\usepackage{graphicx}

\usepackage{amsmath,amssymb, mathabx,verbatim, mathtools}
\usepackage[a4paper, total={6.3in, 9in}]{geometry}
\usepackage{amsfonts}       
\usepackage{multicol}
\usepackage{color}
\usepackage{array}
\usepackage{subcaption}

\usepackage{hyperref}       
\usepackage{url}            
\usepackage{booktabs}       
\usepackage{nicefrac}       
\usepackage{microtype}      
\usepackage{physics}
\usepackage{stackengine}
\usepackage{tcolorbox}      

\usepackage{listings}
\usepackage{bbm}
\usepackage{float}
\usepackage{centernot}

\usepackage{diagbox}

\usepackage[title]{appendix}

\usepackage{natbib}

\usepackage{tikz}
\usepackage{pgfplots}
\usetikzlibrary{shapes}
\usetikzlibrary{plotmarks}
\pgfplotsset{compat=1.16}

\usepackage{atveryend}

\usepackage[utf8]{inputenc}
\DeclareUnicodeCharacter{211D}{$\mathbb{R}$}

\DeclareMathOperator*{\argmax}{arg\,max}
\DeclareMathOperator*{\argmin}{arg\,min}

\newcommand{\RNum}[1]{\uppercase\expandafter{\romannumeral #1\relax}}

\newcommand{\N}{\mathbb{N}}

\newcommand{\E}{\mathbb{E}}

\newcommand{\bbs}{\mathbb{S}}

\newcommand{\T}{^\intercal}

\newcommand{\cc}{\frac{k-\|E\|_K + \eta\sqrt{(k-\|E\|_K)^2 - 4 k\|E\|_K}}{2k}}
\newcommand{\omtan}{\|E\|_{\mathcal{T}_K(\bar{x})}}
\newcommand{\omek}{\|E\|_K}
\newcommand{\bx}{\bar{x}}

\newcommand{\tvt}{\tilde{v}_{\tilde{t}}}

\SetKwInput{KwInput}{Input}                
\SetKwInput{KwOutput}{Output}              
\SetKwRepeat{KwRepeat}{repeat}{until}

\usepackage{amsthm}
\theoremstyle{definition}

\newtheorem{theorem}{Theorem}[section]
\newtheorem{proposition}{Proposition}[section]
\newtheorem{lemma}{Lemma}[section]
\newtheorem{corollary}{Corollary}[section]

\newtheorem{example}{Example}

\title{\textbf{Non-Sparse PCA in High Dimensions via Cone Projected Power Iteration}}

\author{Yufei Yi \\\small\href{mailto:yy544@cmu.edu}{yy544@cmu.edu}
\and Matey Neykov\\\small\href{mailto:mneykov@stat.cmu.edu}{mneykov@stat.cmu.edu}}

\date{%
    Department of Statistics and Data Science, Carnegie Mellon University\\%
    5000 Forbes Ave, Pittsburgh, PA 15213, U.S.A.\\[2ex]%
}

\begin{document}
\maketitle

\begin{abstract}
In this paper, we propose a cone projected power iteration algorithm to recover the first principal eigenvector from a noisy positive semidefinite matrix. When the true principal eigenvector is assumed to belong to a convex cone, the proposed algorithm is fast and has a tractable error.
Specifically, the method achieves polynomial time complexity for certain convex cones equipped with fast projection such as the monotone cone. It attains a small error when the noisy matrix has a small cone-restricted operator norm. We supplement the above results with a minimax lower bound of the error under the spiked covariance model. Our numerical experiments on simulated and real data, show that our method achieves shorter run time and smaller error in comparison to the ordinary power iteration and some sparse principal component analysis algorithms if the principal eigenvector is in a convex cone.
\end{abstract}

\small\textbf{\textit{Keywords---}} Power iteration; Dimension reduction; Principal Component Analysis; Convex cone; Monotone cone.

\section{Introduction}
\label{intro:section}
Principal component analysis was developed by  \citet{hotelling1933analysis} after its origin by \citet{pearson1901liii}, and is widely used nowadays for dimension reduction. It works by replacing a set of $p$ variables with a smaller set of principal components which capture the maximal variance. A principal component is a linear combination of the $p$ covariates. The coefficients of such a linear combination depend on the principal eigenvectors of the population covariance matrix, which is often estimated by eigenvectors of the sample covariance matrix.
Numerical methods to compute principal eigenvectors include the QR algorithm and power iteration \citep[Chapter~5]{watkins2004fundamentals}. 

The consistency of principal component analysis is thoroughly studied in the statistical literature. See for example \citet{paul2007asymptotics,nadler2008finite,bai2010spectral,fan2015asymptotics}. 
In low dimensions $(p \ll n)$, the sample estimators consistently recover the population principal eigenvectors \citep{anderson2003introduction}.
However in high dimensions $(p \gg n)$, all methods fail to recover the population principal eigenvectors if there is no additional structure imposed. \citet{johnstone2009consistency} introduced the spiked covariance model and proved that the sample estimator would not be consistent if $p/n$ is bounded away from zero. \citet[Example~15.19]{wainwright2019high} shows that the minimax risk for estimating the spiked eigenvector is lower bounded by a quantity related to $p/n$. 

In this paper, we study the problem of estimating the principal eigenvector of a positive semi-definite matrix in high dimensions, given that this vector belongs to a known convex cone. 
We view this setup as an alternative to the commonly used sparse PCA. Of course there are many ready-to-use methods available for sparse PCA, but we recognize that the first principal eigenvector might not always be sparse. The idea of considering general convex cone constraints is inspired by the fact that the set of $s$-sparse vectors forms a cone (albeit a non-convex one). We stick with the convex cone constraint for two reasons: there are adequate results in convex analysis literature to support the analysis of our algorithm, and the set of convex cones is general enough to cover some useful cases. One example is the non-negative orthant cone, which arises in neural signal processing \citep{pavlov2007sorting, quiroga2009extracting}, computer vision \citep{lee1999learning, arora2016computing}, and gene expression \citep{lazzeroni2002plaid}. Another example would be the monotone cone. In time series forecasting, where often times more recent observations are more important, imposing monotonicity constraints makes intuitive sense. In addition due to the nature of the monotone cone, vectors can be estimated with high precision even in very high-dimensional settings.

On the algorithmic side, we propose a cone projected power iteration to estimate the first principal eigenvector. This is inspired by the work of \citet{yuan2013truncated}, where the authors proposed a truncated power iteration to find the principal eigenvector assuming it is sparse. The truncated power iteration, may be effectively thought of as a projected power iteration over the cone formed by $s$-sparse vectors. It is thus natural to investigate whether the same approach carries over to the distinct setting of a convex cone constraint. A bulk of the effort in our work is dedicated to making this precise. On the analytical side, we derive the time complexity of the proposed cone projected power iteration algorithm, and provide an upper bound and a lower bound of its estimation error. Some numerical experiments are implemented to address the usefulness of the cone projected power iteration algorithm.

\subsection{Related Work}
With the failure of ordinary principal component analysis in high dimensions, researchers have started to impose additional structure on the eigenvectors. Such a structure is exploited by introducing constraints or penalties.
The simplest and most studied structure is that of sparsity. Methods exploiting the sparsity structure are referred to as sparse principal component analysis. 
As to the origin, \citet{cadima1995loading} first proposed the idea to approximate a given principal component by using only a subset of features.
The first computational technique --SCoTLASS-- was established by \citet{jolliffe2003modified}, which maximizes the variance of a principal component under the $\ell_1$ constraint of eigenvector, inspired by LASSO \citep{tibshirani1996regression}. 
Afterwards, \citet{zou2006sparse} proposed ElasticNet SPCA to regress a principal component on $p$ variables with elastic net constraint \citep{zou2005regularization} to get a sparse eigenvector. \citet{witten2009penalized} established the connections between SCoTLASS and ElasticNet SPCA. In the seminal work of \citet{d2005direct}, the computation of the first sparse eigenvector is formed as approximating the covariance matrix by a rank-one spike matrix under Frobenius norm with constrains on the spike vector. Then it is relaxed to a semidefinite programming problem.
\citet{vu2013fantope} generalized the work of \citet{d2005direct} to compute more than one sparse eigenvectors by incorporating Fantope in the constraint function.
Besides methods based on the relaxation of the sparsity constraint, there is also a substantial literature on non-relaxed optimization techniques. \citet{moghaddam2006spectral} computed the non-zero elements of a sparse eigenvector by solving an unconstrained optimization on corresponding submatrices. The submatrices is selected by bi-directional greedy search.
\citet{johnstone2009consistency} proposed a method which uses the coordinates with highest variance after wavelet transforms of the data. Years later, some new numerical methods came up, such as the work of \citet{yuan2013truncated} which integrated the power iteration algorithm with truncation in each iteration, and \citet{ma2013sparse} which incorporated the QR algorithm with a thresholding step. 
Theoretical analysis of the convergence rate of sparse principal analysis can be found in \citet{birnbaum2013minimax,cai2013sparse,vu2013minimax}.

Despite the pervasive study of sparse constraints in principal component analysis, the study of methods with other constraints are scarce in the literature. Examples of non-sparse constraints include a subspace constraint \citep{de2004learning} and a non-negative orthant cone constraint \citep{montanari2015non}. In addition to the fact that those two types of constraints are special cases of convex cones, our work is quite different from theirs: \citet{montanari2015non} analyze the error rate of the non-negative approximate message-passing algorithm under the spiked covariance model, and \citet{de2004learning} proposes a subspace constrained spectral clustering algorithm without providing error analysis. A recent work from \citet{cai2020optimal} gives statistical analysis of constrained PCA under the matrix denoising model and the spiked Wishart model under very general constraints. In the present paper we use a different loss function from theirs, and we do not assume any statistical model on the data when deriving the upper bound of estimation error. We also propose a practical iterative algorithm, whereas \citet{cai2020optimal} rely on a constrained optimization formulation which may not be easily implementable. Moreover, in order to apply the general lower bound of \citet{cai2020optimal}, one needs to construct local packing sets at different resolutions manually and solve an identity equating the resolution of the packing to the square root of the log of its cardinality. This is quite distinct from our lower bounds, which are also developed under the spiked Wishart model, since we provide a universal lower bound on the packing set we pick (which is different from that of \citet{cai2020optimal}). Furthermore, our lower bound cannot be derived from that of \citet{cai2020optimal} as our loss function is smaller than the one they consider. 

In this work we consider the same type of constraint as the work of \citet{deshpande2014cone} (i.e., convex cone constraints), but our setting is different from theirs.
\citet{deshpande2014cone} consider a spiked Gaussian Wigner model $A = \nu\bx\bx\T + Z$, where $Z$ is a symmetric Gaussian noise matrix \citep[Definition 3.2]{perry2018optimality}. In contrast, in this paper we estimate the principal eigenvector $\bx$ of any population covariance matrix $\bar{A}$ from a positive semidefinite observation $A = \bar{A}+E$, where $E$ is a mean-zero noise matrix. Our main results are not tied to any distributional assumptions on the error matrix, i.e., we provide deterministic inequalities which control the estimation error (in an $L_2$ sense) in terms of certain cone-restricted norms of the error matrix. We then use those general results to analyze the spiked covariance model where the upper and lower bounds we prove nearly match in some special cases of the convex cone constraint, while the work of \citet{deshpande2014cone} did not provide clear connections between the upper and lower bounds. Notice also that spiked Gaussian Wigner model in \citet{deshpande2014cone} clearly differs from the spiked covariance model. Besides the big difference in the settings, we will draw some further comparisons between our results and that of \citet{deshpande2014cone}: (i) our algorithm is cleaner and does not require a tunable hyper-parameter, (ii) our algorithm provably has a finite run-time guarantee, (iii) our upper bounds are strictly sharper than the ones exhibited by \citet{deshpande2014cone}, (iv) our main lower bound result is markedly different from the lower bound result of \citet{deshpande2014cone} which lower bounds the minimax risk without an explicit dependence on the signal strength.

\subsection{Organization}
The remaining part of the paper is structured as follows. Section \ref{preliminary:sec} provides some common notations, sets up problem formally and gives a result on the spiked covariance model with conic constraints. Section \ref{ideal:section} discusses why cone constrained eigenvector estimation is hard to solve in general, and provides error rate guarantees for the so-called idealized estimator. Section \ref{estimation:section} is dedicated to introducing the cone projected power iteration algorithm, which also includes the study of its time complexity, convergence, and upper bound of the error rate. Section \ref{bound:section} provides a lower bound over the convex cone constrained principal eigenvector estimation problem. We compare the obtained upper and lower bound under the spiked covariance model. Some example convex cones are studied to derive more informative upper and lower bound. Section \ref{experiment:section} presents both simulated and real data experiments to support our theoretical findings. Finally, a brief discussion is provided in Section \ref{discussion:sec}.

\section{Notation and Preliminary} \label{preliminary:sec}

In this section we formalize the problem, we outline some notation and provide a result on the spiked covariance model with conic constraints.

\subsection{Problem Formulation}
\label{formulation:section}
In this paper we focus on the following concrete problem. Suppose $\bar{A}$ is a $p\times p$ positive semidefinite matrix with a first principal eigenvector $\bar{x}\in K$ where $K\subset\mathbb{R}^p$ is a known convex cone. Then $\bx$ is the solution of
\begin{align} \label{pca_formula}
         \argmax_{u \in K \bigcap \mathbb{S}^{p-1}} u\T  \bar{A} u.
\end{align}
where $\mathbb{S}^{p-1}$ is the unit sphere in $\mathbb{R}^p$. Instead of observing $\bar{A}$ we assume we get to observe a noisy matrix $A=\bar{A}+E$, where $E$ is the stochastic noise. The problem of interest becomes to recover $\bar{x}$ from a noisy observation $A$, with prior knowledge that $\bar{x}\in K$ and assuming that the noisy matrix $A$ is positive semidefinite. For example, $A$ could be the sample covariance matrix of a data set, and $\bar{A}$ could be its population covariance matrix. 

Let $\hat{v}$ be the estimated principal eigenvector. The loss function we use is $\|\hat{v}-\bx\|_2\wedge \|\hat{v}+\bx\|_2$, which is sign invariant and aligns with the signless nature of an eigenvector. In several places we will also use $\|\hat{v}-\bx\|_2$ if it is assumed or guaranteed that $\hat{v}\T\bx \geq 0$. In passing, we also remark that there are other popular loss functions studied in PCA literature such as the square of the sine of the angle $1-(\hat{v}\T\bx)^2$, or the projector distance $\|\hat{v}\hat{v}\T- \bx\bx\T\|_F$. It's not hard to show that $\|\hat{v}\hat{v}\T- \bx\bx\T\|_F^2=2-2(\hat{v}\T\bx)^2$, and both of these loss functions are in general strictly larger than $\|\hat{v}-\bx\|^2_2\wedge \|\hat{v}+\bx\|^2_2 = 1 - |\hat{v}\T\bx|$. However, it is also clear that for values $|\hat{v}\T\bx|$ close to $1$, the loss functions are of the same order. The latter observation implies that any small enough upper bound on our loss function implies a corresponding upper bound on the sine of the angle and the projector distance loss functions.

\subsection{Notation}
We now outline some commonly used notation. Let $\lambda, \,\mu$ denote the largest and second largest eigenvalues of $\bar{A}$ and let $\nu:= \lambda - \mu$ be the first eigengap of $\bar{A}$. 

Given a cone $C\subset\mathbb{R}^p$, a cone-restricted operator norm of a $p\times p$ matrix $E$ is defined as 
\begin{align*}
    \|E\|_C &= \sup\limits_{x,y\in C \bigcap \mathbb{S}^{p-1}}|x\T Ey|,
\end{align*}
Notice that $\|E\|_C$ constitutes a seminorm, and one trivially has $\|E\|_C \leq \|E\|_{op}$ where $\|E\|_{op}$ is the operator norm of $E$. For completeness we mention that other forms of cone-restricted norms have also appeared in conic optimization literature \citep[Section 1.2]{amelunxen2014gordon}.  

For a set $T\subset\mathbb{R}^p$, define the Gaussian complexity of $T$ as 
\begin{align*}
    w(T)& = \mathbb{E}\sup\limits_{t\in T}\,\,\langle g,t\rangle, \quad\text{where}\,\,g \sim \mathcal{N}(0,I_p),
\end{align*}
which is the expectation of maximum magnitude of the canonical Gaussian process on $T$. The Gaussian complexity of set $T$ can be upper bounded in terms of the metric entropy of $T$ via Dudley's integral inequality \citep[Theorem 8.1.10]{vershynin2018high}. Gaussian complexity is a basic geometric quantity measuring the size of the set similar to volume or diameter. It has several nice properties such as invariance under affine transformations, respecting Minkowski sums and scalar multiplication, and is related to the value of the diameter of the set $T$ via (dimension dependent) upper and lower bounds \citep[see Proposition 7.5.2]{vershynin2018high}.

The tangent cone of a convex cone $K\subset\mathbb{R}^p$ at $\bx\in K$ consists of all the possible directions from which a sequence in $K$ can converge to $\bx$. It is defined as 
\begin{align*}
    \mathcal{T}_K(\bx) & = \{ t(v-\bx)\,\,:\, t\geq 0, v\in K\}. 
\end{align*}

Throughout we use $\|\cdot\|$ as a shorthand for the Euclidean norm $\|\cdot\|_2$. The projection of a vector $v \in \mathbb{R}^p$ onto a convex cone $K\subset\mathbb{R}^p$ is defined as 
\begin{align*}
    \Pi_K v &= \argmin\limits_{x\in K} \|v-x\|.
\end{align*}

We use $\lesssim$ and $\gtrsim$ to mean $\leq$ and $\geq$ up to positive universal constants. Next we define two constants $c_{-1}$ and $c_1$ which will be used in the statements of our results in the consequent sections. For precise definitions in terms of the eigengap $\nu$ and $\|E\|_K$ please refer to Appendix \ref{app:theorem}. Here we only mention that $c_{-1}$ and $c_1$ are well defined when the eigengap $\nu \gtrsim \|E\|_K$ and that $\frac{\|E\|_K}{\lambda} \leq c_{-1} \lesssim \frac{\|E\|_K}{\nu}$ and $c_1 \geq \frac{2}{5}$.
We use $\wedge $ and $\vee$ as a shorthand for the $\min$ and $\max$ of two numbers respectively. 
Finally we will sometimes use the convenient shorthand $[n] = \{1,2,\ldots, n\}$ for an integer $n \in \mathbb{N}$.

\subsection{The Spiked Covariance Model}
In this subsection we introduce the spiked covariance model and present a result controlling the expectation of the cone-restricted operator norm of its error matrix. This will be useful later on to specify our general results to the spiked covariance model. 

Let us observe $n$ i.i.d. samples from the model $X_i \sim \mathcal{N}(0, I + \nu \bar{x}\bar{x}\T)$, where $\bar{x}\in \mathbb{S}^{p-1}$ and $i = 1,2,\ldots,n$.  This model is known as the spiked covariance model, and has been intensely studied in the literature. It was first considered by \citet{johnstone2009consistency}. The sample covariance matrix is constructed as $A = n^{-1}\sum_{i = 1}^n X_i X_i\T$, and the population covariance matrix is $\bar A = I + \nu\bx\bx \T$. Clearly then, the noise matrix equals to $E = A - \bar{A}$. The parameter $\nu$, which is the first eigengap of the population covariance matrix $\bar A$, can be understood as signal strength in the model. If one assumes that $\bar{x} \in K$, the spiked covariance model is an ideal toy model to analyze and compare the performances of competing estimators. For this reasons we will come back to it repeatedly throughout our paper. 

Lemma \ref{tan_gw} below shows an upper bound on $\E \|E\|_K$ in terms of the Gaussian complexity $w(K\bigcap\mathbb{S}^{p-1})$ under the spiked covariance model. The proof of Lemma \ref{tan_gw} is established by applying a powerful empirical process upper bound \citep[Theorem A]{mendelson2010empirical} which upper bounds the expectation of the supremum of the quadratic process in terms of Talagrand's $\gamma_2$ function and an Orlicz norm. The result of \citet[Lemma 3.2.3]{vu2012minimax} is similar to Lemma \ref{tan_gw}, but in the sparse constraint setting. The bound in Lemma \ref{tan_gw} is also tighter than that of \citet[Lemma 3.2.3]{vu2012minimax} in terms of the signal strength $\nu$, since if one uses the approach of \citet[Lemma 3.2.3]{vu2012minimax} in our setting one obtains a rate $\mathcal{O}\big((\nu+1)[
        \frac{w(K\bigcap\mathbb{S}^{p-1})}{\sqrt{n}} \vee 
        \frac{w^2(K\bigcap\mathbb{S}^{p-1})}{n}]\big)$ which is sub-optimal.
\begin{lemma}[Upper Bound on $\|E\|_K$ under Spiked Covariance Model]
\label{tan_gw}
Suppose we have $n$ i.i.d. observations from the spiked covariance model $X_i \sim \mathcal{N}(0, I + \nu \bar{x}\bar{x}\T)$, where $\bar{x}\in \mathbb{S}^{p-1}$. Let $A = n^{-1}\sum_{i = 1}^n X_i X_i\T$ be the sample covariance matrix, $\bar A = I + \nu\bx\bx \T$ be the population covariance matrix, and $E = A - \bar{A}$ be the noise matrix. For any convex cone $K\subset \mathbb{R}^p$ we have
\begin{align*}
    \mathbb{E}\omek \lesssim 
        \sqrt{\nu+1}\bigg[ \frac{w(K\bigcap\mathbb{S}^{p-1})}{\sqrt{n}} \vee \frac{w^2(K\bigcap\mathbb{S}^{p-1})}{n} \bigg] + \frac{\nu+3-3\sqrt{\nu+1}}{\sqrt{n}}
\end{align*}
\end{lemma}

In the next section we formalize and analyze the ``idealized'' estimator for cone constrained eigenvector estimation. While being difficult to compute in some cases, this estimator is a very intuitive way to estimate the principal eigenvector $\bar{x}$. We will use Lemma \ref{tan_gw} in order to illustrate the performance of the idealized estimator on the spiked covariance model.

\section{The Idealized Estimator}
\label{ideal:section}
We start with an intuitive expression of the constrained eigenvector estimation problem, and discuss why it is hard to solve. There exists a natural estimator to problem (\ref{pca_formula}) which simply plugs in the observed matrix $A$ instead of the target matrix $\bar A$. In particular consider estimating $\bar{x}$ with 
\begin{align}\label{idealized:estimator}
    v \in \argmax_{u \in K \bigcap \mathbb{S}^{p-1}} u\T  A u,
\end{align} 
where $A$ is the observed noisy matrix. We refer to an estimator $v$ as an idealized estimator since the above program is non-convex and could be NP-hard to solve. For example, if $K$ is the non-negative orthant $\{\mathbf{v}\in\mathbb{R}^p: 0 \leq v_j,\,\forall j\in [p]\}$, then solving \eqref{idealized:estimator} reduces to a copositive program by rewriting the quadratic form to a trace function \citep[Exercise 7.4]{gartner2012approximation}. Copositive programming is NP-hard since the maximum clique problem is equivalent with a copositive program \citep{dur2010copositive}. Of course, in some cases the idealized estimator is tractable: if $K$ is an $s$-dimensional subspace, the problem $\argmax_{u \in K \bigcap \mathbb{S}^{p-1}} u\T Au$ reduces to an unconstrained eigenvector estimation in a lower dimension $s$, which can be solved in polynomial time \citep{garber2015fast}.

Even though the idealized estimator might be impractical to compute, we are still able to provide a bond on its $L_2$ error (see Theorem \ref{cvx_est} below). Notice that the idealized estimator is not generally guaranteed to have a positive dot product with the true eigenvector of $\bar{A}$, so we consider both situations: $v\T  \bar{x} \geq 0$ and $v\T \bar{x} \leq 0$.
\begin{theorem}[$L_2$ Error Rate of the Idealized Estimator]
\label{cvx_est}
For any $v$ as a solution in \eqref{idealized:estimator}, we either have 
\begin{align*}
    \| v-\bar{x}\| &\leq \sqrt{\frac{4\omek}{\nu}}\wedge\frac{8\omtan}{\nu}\quad\text{for $\,\,v\T\bar{x} \geq 0$, or}\\
    \| v+\bar{x}\| &\leq \sqrt{\frac{4\omek}{\nu}}\wedge\frac{8\omek}{\nu}\quad\text{for}\,\,v\T\bar{x} \leq 0.
\end{align*}
\end{theorem}

It is straightforward to see that the upper bounds in Theorem \ref{cvx_est} are naturally related to the operator norm of the noise matrix $E$. For completeness we state this in the following Corollary \ref{cvx_est_coro}.
\begin{corollary}[$L_2$ Error Rate of the Idealized Estimator] For the estimate defined in Theorem \ref{cvx_est} we have
\label{cvx_est_coro}
\begin{align*}
    \| v-\bx\| \wedge  \| v+\bx\| \leq 
    \sqrt{\frac{4\|E\|_{op}}{\nu}}
    \wedge 
    \frac{8\|E\|_{op}}{\nu}.
\end{align*}
\end{corollary}
One should keep in mind however, that the conclusion of Theorem \ref{cvx_est} could be much tighter than Corollary \ref{cvx_est_coro} when $\mathcal{T}_K(\bx)$ and/or $K$ is much smaller than $\mathbb{R}^p$. 

We will now combine the results of Theorem \ref{cvx_est} and Lemma \ref{algo1_error_spiked} in order to provide a guarantee for the performance of the idealized estimator in the spiked covariance model. We start by taking expectation and using Jensen's inequality to obtain
\begin{align*}
    \mathbb{E}\| v-\bar{x}\| \leq \sqrt{\frac{4\mathbb{E}\omek}{\nu}}\wedge\frac{8\mathbb{E}\omtan}{\nu}\quad\text{for $\,\,v\T\bar{x} \geq 0$}.
\end{align*}
The expectations $\mathbb{E}\omek$ and $\mathbb{E}\omtan$ can be bounded in terms of Gaussian complexity $w(K\bigcap\mathbb{S}^{p-1})$ and $w(\mathcal{T}_K(\bx)\bigcap\mathbb{S}^{p-1})$ using Lemma \ref{tan_gw}. For simplicity of exposition suppose that the convex cone $K$ satisfies $w(K\bigcap\mathbb{S}^{p-1})<\sqrt{n}$ and $w(\mathcal{T}_K(\bx)\bigcap\mathbb{S}^{p-1})<\sqrt{n}$. This is often times a reasonable assumption provided that $p$ is not too large. For example a monotone cone $K$ satisfies $w(K\bigcap\mathbb{S}^{p-1})\approx\sqrt{\log p}$ \citep[Section~D.4]{amelunxen2014living}, and $w(\mathcal{T}_{K}(\bx)\bigcap\mathbb{S}^{p-1}) \approx \sqrt{m\log \frac{ep}{m}}$ where $m$ is the number of constant pieces in $\bx$ \citep[Proposition~3.1]{bellec2018sharp}. 
We also assume that the eigengap $\nu$ doesn't scale with $n$. 
Then the above upper bound reduces to 
\begin{align}
\label{ideal_upper1}
    \mathbb{E}\|v-\bar{x}\|
     \lesssim \sqrt{\frac{w(K\bigcap\mathbb{S}^{p-1})}{(\nu\wedge\sqrt{\nu})\sqrt{n}}} \wedge \frac{w(\mathcal{T}_K(\bx)\bigcap\mathbb{S}^{p-1})}{(\nu\wedge\sqrt{\nu})\sqrt{n}}
     \quad\text{for $\,\,v\T\bar{x} \geq 0$}
\end{align}

Similarly, under the same assumptions one can show that
\begin{align}
\label{ideal_upper2}
     \mathbb{E}\|v+\bar{x}\|
     \lesssim \frac{w(K\bigcap\mathbb{S}^{p-1})}{(\nu\wedge\sqrt{\nu})\sqrt{n}}\quad\text{for}\,\,v\T\bar{x} \leq 0
\end{align}


\section{Cone Projected Power Iteration}
\label{estimation:section}
In this section we introduce a computationally tractable algorithm -- the cone projected power iteration -- to solve the constrained eigenvector estimation problem \eqref{pca_formula}. We also demonstrate that our algorithm also possesses similar estimation guarantees to the idealized estimator considered in the previous section.

\subsection{Algorithm}
\label{sec:algo}
In the absence of constraints, solving \eqref{pca_formula} would have been equivalent to finding the principal eigenvector of $A$. One way to find the principal eigenvector is to use the power iteration algorithm, which starts with a vector $v_0$, such that $v_0$ has a non-zero dot product with the principal eigenvector of $A$, and iterates the following recursion $v_t = \frac{A v_{t - 1}}{\|A v_{t - 1}\|}$ for $t = 1, 2, \ldots$. We now suggest a simple modification of the ordinary power iteration algorithm, to target the constrained problem (\ref{pca_formula}). We modify the power iteration by adding a projection step in each iteration in order to force the algorithm to choose vectors belonging to the set $K \bigcap \mathbb{S}^{p-1}$.

\vskip 0.5cm
\resizebox{14cm}{!}{%
\centering
\begin{algorithm}[H]
\caption{Cone Projected Power Iteration Single Vector Version}
\label{algo1}
\KwInput{$\Delta\in\mathbb{R}$ stopping criteria, $K$ a convex cone, $v_0$ initialization}
\KwOutput{$v_{out}$}
\KwData{$A\in\mathbb{R}^{p\times p}$ positive semi-definite matrix}

\KwRepeat{$\|v_{t+1} - v_{t}\|\leq\Delta$}{$v_{t + 1} \leftarrow \frac{\Pi_KAv_{t}}{\|\Pi_KAv_{t}\|}$\\
$t++$}
$v_{out}\leftarrow v_{t+1}$
\end{algorithm}
}
\vskip 0.5cm

To achieve the consistency of Algorithm \ref{algo1}, $v_0\T \bar{x} \geq c_0 > 0$ for some $c_0$ will be required (see Theorem \ref{stop} below). Sometimes it may be more convenient to assume that one can find a vector $v_0$ for which it is only known that $|v_0\T \bar{x}| \geq c_0 > 0$, since the sign of $\bar{x}$ is unknown. To facilitate this assumption, we suggest Algorithm \ref{algo2} which runs the procedure of Algorithm \ref{algo1} two times, once starting with $v_0$ and once with $-v_0$, and returns the vector $v$ corresponding to the larger quadratic form $v\T A v$. The motivation of this idea is clear: for at least one of the two starts $v_0$ or $-v_0$ we will have a dot product with $\bar{x}$ which is bigger than $c_0$. Next, we hope that the output vector $v$ which has a bigger product $v\T A v$ will be the one that has started with a positive dot product with $\bar{x}$. Theorem \ref{bad_vt} shows that even if this is not the case, the fact that the product $v\T A v$ is larger gives us a leverage on the final output vector $v$. This vector will be close to $\bar{x}$ in either case.

\vskip 0.5cm
\resizebox{14cm}{!}{%
\centering
\begin{algorithm}[H]
\caption{Cone Projected Power Iteration Double Vectors Version}
\label{algo2}
\KwInput{$\Delta\in\mathbb{R}$ stopping criteria, $K$ a convex cone, $v_0$ initialization}
\KwOutput{$v_{out}$}
\KwData{$A\in\mathbb{R}^{p\times p}$ positive semidefinite matrix}

$v_+$ = Output of Algorithm \ref{algo1} initialized with $v_0$\\
$v_-$ = Output of Algorithm \ref{algo1} initialized with $-v_0$\\
$v_{out} \leftarrow \argmax_{v\in \{v_+, v_-\}} v\T  Av$
\end{algorithm}
}
\vskip 0.5cm

\subsection{Convergence}
Proposition \ref{algo_complexity} below, gives the number of iterations needed to achieve $\|v_{t} - v_{t-1}\|\leq \Delta$, and implies that Algorithm \ref{algo1} converges. Since the computing time of Algorithm \ref{algo2} is at most twice that of Algorithm \ref{algo1}, Algorithm \ref{algo2} converges as well. The proof of Proposition \ref{algo_complexity} also reveals that the sign of $v_t\T v_{t-1}$ is positive and never flips, so that the value of $\|v_{t} - v_{t-1}\|$ remains  smaller than $\sqrt{2}$.

\begin{proposition}[Algorithm \ref{algo1} Time Complexity]
\label{algo_complexity}
For Algorithm \ref{algo1}, to get 
$\|v_{t}-v_{t-1
}\|\leq \Delta$, we need at most $\Big\lceil \frac{\log[\frac{\lambda+\omek}{v_0\T Av_0}]}{\log(1+\Delta^2)}\Big\rceil$ iterations, assuming that $v_0\T  A v_0 > 0$.
\end{proposition}
\begin{proof} The proof will proceed in two steps.
\begin{enumerate}
    \item We will first argue that $\frac{v_{t+1}\T Av_{t+1}}{v_t\T Av_t}\geq 1+ \|v_{t+1}-v_t\|^2$, so that $\{v_t\T Av_t\}$ is an increasing series. By Moreau's decomposition (see Theorem \ref{moreausthm} in the Appendix) and the identity $\|\Pi_KAv_{t-1}\| = v_t\T Av_{t-1}$ (see 1. of Lemma \ref{ineqs} in the Appendix) we have,
    \begin{align*}
        v_t\T v_{t+1} = \frac{v_t\T \Pi_K Av_t}{\|\Pi_K Av_t\|} \geq \frac{v_t\T  Av_t}{v_{t+1}\T Av_t}
    \end{align*}
    Then 
    \begin{align*}
        v_{t+1}\T Av_{t+1} - v_t\T Av_t& = 2\langle Av_t, v_{t+1} - v_t\rangle + 
        (v_{t+1} - v_t)\T  A (v_{t+1} - v_t) \\
        & \geq 2\langle Av_t, v_{t+1} - v_t
        \rangle \\
        & \geq 2 v_{t+1}\T  Av_t(1-v_t\T v_{t+1}) 
        \\
        & \geq v_t\T  Av_t\|v_{t+1} - v_t\|^2
    \end{align*}
    From the above inequality and the fact that $v_0\T A v_0 > 0$, it follows that $v_t\T A v_t > 0$ for all $t$. Hence the claimed inequality follows by rearranging terms. Since $\{v_t\T Av_t\}$ is strictly increasing, and bounded by the first principal eigenvalue of $A$, the algorithm converges. At convergence, $\{v_t\T Av_t\}$ is stationary.
    \item Suppose $\|v_{t}-v_{t-1}\|\geq \Delta$ for $t \leq n$, then 
    \begin{align*}
        \frac{v_n\T Av_n}{v_0\T Av_0}  \geq (1+\Delta^2)^n 
        \,\,\,\, 
        \Rightarrow 
        \,\,\,\, 
        n  \leq \frac{\log[\frac{\lambda+\omek}{v_0\T Av_0}]}{\log(1+\Delta^2)}
    \end{align*}
\end{enumerate}
\end{proof}

For certain types of convex cones such as the monotone cone or the positive orthant, the time complexity of Algorithm \ref{algo1} can be calculated explicitly. The complexity of computing $Av_t$ is $\mathcal{O}(p^2)$. The projection onto a monotone cone can be obtained by isotonic regression \citep{barlow1972statistical}, which takes $\mathcal{O}(p)$ flops using the pool adjacent violators algorithm \citep[Page~9]{mair2009isotone}. Thus the complexity of one iteration is $\mathcal{O}(p^2)$. 
For the number of iterations needed, by Lemma \ref{tan_gw} under the spiked covariance model, the upper bound of $\omek$ is $\sqrt{\nu+1}(\sqrt{\frac{\log p}{n}} \vee \frac{\log p}{n})$, given the Gaussian complexity of a $p$-dimensional monotone cone is $\sqrt{\log p}$ \citep[Section~D.4]{amelunxen2014living}. Since $v_0\T Av_0$ and $\Delta$ are constants controlled by user, the number of iterations is of the order $\mathcal{O}\Big(\log[\lambda+\sqrt{\nu+1}(\sqrt{\frac{\log p}{n}} \vee \frac{\log p}{n})]\Big)$. The overall time complexity of Algorithm \ref{algo1} for monotone cone is polynomial, and Algorithm \ref{algo2} retains the same order of complexity since it just applies Algorithm \ref{algo1} twice. Similarly one can show that the number of iterations needed for the positive orthant case is $\mathcal{O}(\log[\lambda+\sqrt{\nu+1}\,(p/n)])$ since the order of $\|E\|_K$ in this case is $\sqrt{\nu+1}\,(p/n)$.

\subsection{Upper Bound}
The cone projected power iteration algorithm attempts to solve problem \eqref{pca_formula} in finite time, but how precise is its estimation? In this section we give some statistical guarantees of Algorithm \ref{algo1} and Algorithm \ref{algo2}. A consistent estimation is achieved if the eigengap $\nu$ is larger than a certain threshold, the stopping criterion $\Delta$ is small enough, and the starting vector $v_0$ is properly chosen. 

In Theorem \ref{stop}, we provide an upper bound of estimation error of the cone projected power iteration algorithm, which is valid for all positive semidefinite input matrices. Theorem \ref{stop} gives error rates which coincide exactly with the rates of the idealized estimator in Theorem \ref{cvx_est} if one assumes that its output has a positive dot product with the target vector $\bar{x}$. Achieving the same error rate as the idealized estimator, Algorithm \ref{algo1} only requires finite time to run as proved in Proposition \ref{algo_complexity}.

\begin{theorem}[$L_2$ Error Rate of Algorithm \ref{algo1}]
\label{stop}
To solve the problem \eqref{pca_formula}, we apply Algorithm \ref{algo1} on the matrix $A$, and get the output vector $v_t$. Suppose that the initial vector $v_0$ satisfies $v_0\T \bar{x}\geq c_0 > c_{-1}$. Assume further that $\nu \geq (3 + 2\sqrt{2})\omek$. If the  stopping criterion satisfies $\Delta \leq \min 
\Big\{
\frac{5\omtan}{4(c_0\wedge c_{1})\nu},
\frac{4\omtan^2}{(c_0\wedge c_{1})\lambda \nu},
\frac{\omek}{2\lambda},
1
\Big\}$, 
we have 
\begin{align}\label{conclusion:of:thm:stop}
\|v_t-\bar{x}\|\wedge\|v_t+\bar{x}\|\leq
\sqrt{\frac{8\omek}{(c_0\wedge c_{1})\nu}}
\wedge
\frac{41\omtan}{(c_0\wedge c_{1})\nu}
\end{align}
\end{theorem}
The upper bound in Theorem \ref{stop} can be evaluated in terms of the Gaussian complexity of set $K$ and $\mathcal{T}_K(\bx)$ under the spiked covariance model. With the same derivation and assumptions after Theorem \ref{cvx_est}, we obtain:
\begin{align}
    \mathbb{E}\,\big[\|v_t-\bar{x}\|\wedge\|v_t+\bar{x}\|\big]
     \lesssim \sqrt{\frac{w(K\bigcap\mathbb{S}^{p-1})}{(\nu\wedge\sqrt{\nu})\sqrt{n}}} \wedge \frac{w(\mathcal{T}_K(\bx)\bigcap\mathbb{S}^{p-1})}{(\nu\wedge\sqrt{\nu})\sqrt{n}}.
\label{algo1_error_spiked}
\end{align}

Next we study the estimation error of Algorithm \ref{algo2}. Recall that Algorithm \ref{algo2} runs the cone projected power iteration twice, once with an initial vector $v_0$ and once with $-v_0$. Here we have no knowledge about $v_0$, so we do not know whether $v_0$ or $-v_0$ will have a positive dot product with $\bx$. Without loss of generality we assume $v_0$ to be the ``better'' initialization such that $v_0\T \bar{x} \geq c_0$. Let $v_+$ be the resulting vector started with $v_0$, and $v_-$ be the resulting vector started with $-v_0$. If $v_-\T Av_- \leq v_+\T Av_+$, Algorithm \ref{algo2} will output $v_+$, then Theorem \ref{stop} guarantees that $\|v_+ - \bar{x}\|$ is small. Otherwise the output will be $v_-$. Theorem \ref{bad_vt} ensures that $\|v_- - \bar{x}\| \wedge \|v_- + \bar{x}\|$ is small based on the fact that $\|v_+ - \bar{x}\|$ is small.
\begin{theorem}[$L_2$ Error Rate of Algorithm \ref{algo2} with Bad Initialization]
\label{bad_vt}
To solve the problem \eqref{pca_formula} we apply Algorithm \ref{algo2} on the matrix $A$. Without loss of generality let $v_0\T \bar{x} \geq c_0$, and suppose that $c_0 > c_{-1}$. Assume additionally that $\nu \geq (3 + 2\sqrt{2})\omek$. If $v_-\T A v_- \geq v_+\T Av_+$, the output of Algorithm \ref{algo2} will be $v_-$, but the $L_2$ error will still be bounded as
\begin{align*}
\|v_- -\bx\|\wedge \|v_- +\bx\|
\leq B_1 \wedge B_2
\end{align*}
where 
\begin{align*}
    B_1 & = \sqrt{\frac{2\lambda}{\nu}}\|v_+-\bx\|+
    \sqrt{\frac{4\omek}{\nu}}; \\
    B_2 & = \Big( 
    \frac{8\omek}{\nu} \vee 
    \frac{8\omtan}{\nu}
    \Big) + 
    \sqrt{\frac{2\lambda\|v_+-\bx\|^2 + 8\|v_+-\bx\|\omtan}{\nu}}.
\end{align*}
\end{theorem}
In fact a simple calculation shows that
\begin{align}\label{weak_ineq}
    \|v_- - \bar{x}\| \wedge \|v_- + \bar{x}\| \lesssim \sqrt{\|E\|_K} \wedge (\|E\|_K \vee \omtan),
\end{align}
where the sign $\lesssim$ here means $\leq$, but ignores constants that may depend on $\lambda, \mu, c_0 \wedge c_1$ for simplicity. Observe that the upper bounds for the idealized estimator in Theorem \ref{cvx_est} can be reduced to a similar form as \eqref{weak_ineq}. 


The proof of Theorem \ref{bad_vt} does not rely on the fact that we have two output vectors $v_+$ and $v_-$ to analyze, but utilizes the fact that the final vector has a bigger product with the matrix $A$. Hence, the proof immediately extends to situations where one has multiple starting vectors $v$ (such that $v\T A v > 0$) to pick from. In particular, suppose that one has a $\delta$-covering set of the set $K\bigcap \mathbb{S}^{p-1}$ in $\|\cdot\|$, denoted with $N_2(\delta, K \bigcap \mathbb{S}^{p-1})$, for some fixed $\delta > 0$. Then we know there exists a vector $v_0 \in N_2(\delta, K\bigcap \mathbb{S}^{p-1})$ such that $v_0\T \bar{x} = 1 - \frac{\|v_0 - \bar{x}\|_2^2}{2} \geq 1 - \frac{\delta^2}{2}$. Supposing that $\delta$ is small enough so that $1 - \frac{\delta^2}{2} > c_{-1}$, we can then run Algorithm \ref{algo1} on all vectors  $v \in N_2(\delta, K\bigcap \mathbb{S}^{p-1})$ such that $v\T A v > 0$ (it can be shown that under the assumptions of Theorem \ref{bad_vt}, $v_0$ is necessarily such a vector, see the Appendix for a short proof) and output the vector $v_{out}$ with the largest $v\T A v$. Then either inequality \eqref{conclusion:of:thm:stop} or \eqref{weak_ineq} will hold for $v_{out}$. In both cases $v_{out}$ will be close to $\bar{x}$. The bottleneck of this approach is that it requires looping over all vectors in the set $N_2(\delta, K\bigcap \mathbb{S}^{p-1})$. Clearly $N_2(\delta, K\bigcap \mathbb{S}^{p-1}) \leq N_2(\delta, \mathbb{S}^{p-1}) \leq (1 + \frac{2}{\delta})^p$ \cite[see][Example~5.2]{wainwright2019high} and therefore such an approach requires exponentially many runs in the worst case, but could be tractable for some smaller sets $K$.  On larger sets such strategy may be impractical. However, for some large dimensional constrained problems, one can indeed implement the strategy in polynomial time, provided that $\omega(K \bigcap S^{p-1})$ is small. Indeed, by Sudakov's minoration principle we know that $\sqrt{\log N_2(\delta, K\bigcap \mathbb{S}^{p-1})} \leq \frac{\omega(K \bigcap S^{p-1})}{\delta}$. Since $\delta$ of constant order suffices for strong estimation guarantees of our procedures in many practical applications, one has that if $\omega(K \bigcap S^{p-1}) \lesssim \sqrt{\log p}$ one needs to obtain only polynomially many points in the packing set (assuming $\delta$ is of constant order). Two examples when such a condition is met are the positive monotone cone and the monotone cone (see also Section \ref{mnt_cone:subsec}). One challenge with this approach would be constructing an $\delta$-packing algorithmically.

\section{Statistical Guarantees of Cone Constrained Eigenvector Estimation}
\label{bound:section}
To complete the analysis we give a lower bound of the cone constrained eigenvector estimation problem under the spiked covariance model, and compare the upper and lower bound in both general form and specific cases.

\subsection{Lower Bound}
\label{lower:subsec}
We will now show a lower bound on the principal eigenvector estimation error under the spiked covariance model. We use Fano's method to link the minimax error to the metric entropy of set $K$. Then we are able to evaluate the minimax error in terms of the Gaussian complexity $w(K\bigcap \mathbb{S}^{p-1})$ by using the reverse Sudakov's inequality to bound the metric entropy from below by the Gaussian complexity. This approach results in an extraneous $\log p$ factor in the denominator which may be sub-optimal but we do not presently know of a way to remove this factor.
\begin{theorem}[Minimax Lower Bound under Spiked Covariance Model]Suppose we have $n$ i.i.d. observations $X_1,...,X_n$ where $X_i\sim \mathcal{N}(0, I + \nu\bar{x}\bar{x}\T )$. Let $\bar{x} \in K$ where $K\subset\mathbb{R}^p$ is a convex cone with $w(K\bigcap\mathbb{S}^{p-1})\geq 64\sqrt{\log 3}$.
There exists a constant $C^*(K) \leq 8p$ such that for any $n$ and $\nu$ satisfying $n (\nu \wedge \nu^2) = C^*(K)$ the minimax risk of estimations of $\bx$ based on $X_1,...,X_n$ is bounded below by
\begin{align}\label{main:lower:bound}
    \inf_{\hat v} \sup_{\bx \in K \bigcap \mathbb{S}^{p-1}} \mathbb{E} \|\hat v - \bx\| \gtrsim \frac{w(K\bigcap\bbs^{p-1})}{\log p(\nu\wedge\sqrt{\nu})\sqrt{n}}
\end{align}
\label{minimax}
\end{theorem}
%
We note that our lower bound is slightly atypical, as it does not hold for any pair of sample size and signal $(n,\nu)$. However, our bound does show that for any fixed value of the sample size $n$ there exists a signal strength $\nu$ for which this bound holds, and conversely for any fixed signal strength $\nu$, there exist a sample size $n$ for which the bound holds. The reason for this requirement is our desire to prove this bound for general convex cones, which prevents us from ``localizing'' it, which is a trick allowing one to claim the bound for all possible pairs $(n,\nu)$.  In addition, it is useful to note that our lower bound does not have $\frac{w(K\bigcap\bbs^{p-1})}{\log p(\nu\wedge\sqrt{\nu})\sqrt{n}}\wedge1$ in the left hand side of \eqref{main:lower:bound}, since the the value $C^*(K) = n (\nu \wedge \nu^2)$  guarantees that the lower bound in Theorem \ref{minimax} remains smaller than 1. Moreover, \eqref{main:lower:bound} is a lower bound regarding the quantity $\|\hat v - \bx\|$ and not $\|\hat v - \bx\|\wedge \|\hat v + \bx\|$. The latter can be smaller in principle. By carefully inspecting the proof of Theorem \ref{minimax} one can see that if $K'\subseteq K$ such that all vector pairs in $K'$ have a positive dot product, a lower bound regarding $\|\hat v - \bx\|\wedge \|\hat v + \bx\|$ can be obtained.
\begin{corollary}
In the setting of Theorem \ref{minimax}, for a set $K'\subseteq K$ such that all vector pairs in $K'$ have a positive dot product, the minimax risk of estimations of $\bx$ based on $X_1,...,X_n$ is bounded below by
$$\inf_{\hat v} \sup_{\bx \in K \bigcap \mathbb{S}^{p-1}} \mathbb{E} \big[\|\hat v - \bx\|\wedge\|\hat v + \bx\|\big] \gtrsim \frac{w(K'\bigcap\bbs^{p-1})}{\log p(\nu\wedge\sqrt{\nu})\sqrt{n}}$$
\label{minimax_coro}
\end{corollary}


By comparing the upper bound of the idealized estimator in \eqref{ideal_upper1} and \eqref{ideal_upper2} or the upper bound of the cone projected power iteration in \eqref{algo1_error_spiked} with lower bound in Corollary \ref{minimax_coro}, we see that the upper bound is in general moderately larger than the lower bound by a square root and a $\log p$ factor. The $\log p$ factor here is likely and artifact of our proof, which relies on reverse Sudakov minoration. Since we attempt to solve the problem in a very large generality (for any convex cone), our conjecture of the difference between lower and upper bound is that for specific convex cones there might exist more efficient packing sets than the one used in the proof of the lower bound above, thus yielding tighter lower bounds. On the other hand there could also exist more accurate algorithms specialized to the specific cone of interest. Below we will give more detailed comparisons of the upper and lower bounds for some specific examples, and will provide a separate lower bound for the monotone cone. We will illustrate that the seeming sub-optimality of the lower bound is more subtle in these examples, as we will argue that the lower bound and upper bounds nearly match. 

\subsection{Examples}

\label{case_studies:sec}
Based on the general upper and lower bounds on the estimation error in the previous sections, we study the error rate under three well-structured convex cones, and make comparisons between the upper and lower bounds. Due to the tractable Gaussian complexity of convex cones in those examples, we are able to evaluate the upper and lower bounds only in terms of the sample size $n$, dimension $p$ and eigengap $\nu$, under the spiked covariance model. 

\subsubsection{Non-Negative Orthant}
\label{non_neg:subsec}
The non-negative orthant $K^+ = \{(v_1,...,v_p)\T\in\mathbb{R}^p: 0 \leq v_j,\,\forall j\in [p]\}$ is a convex cone. For the tangent cone $\mathcal{T}_{K^+}(\bx)$ we have $\mathcal{T}_{K^+}(\bx)\subseteq\mathbb{R}^p$, so that $\mathcal{T}_{K^+}(\bx)\bigcap\mathbb{S}^{p-1}\subseteq\mathbb{S}^{p-1}$. Using the fact that $w(\mathbb{S}^{p-1})\asymp\sqrt{p}$ \citep[Example 7.5.7]{vershynin2018high}, we have $w(\mathcal{T}_{K^+}(\bx)\bigcap\mathbb{S}^{p-1})\lesssim\sqrt{p}$. We also derive the order of  $w(K^+\bigcap\mathbb{S}^{p-1})\asymp\sqrt{p}$ in the following Lemma \ref{gw_nonneg:lemma}.
\begin{lemma}
\label{gw_nonneg:lemma}
Let $K^+$ be the $p$-dimensional non-negative orthant $\{(v_1,...,v_p)\T\in\mathbb{R}^p: 0 \leq v_j,\,\forall j\in [p]\}$, and $\mathbb{S}^{p-1}$ be the $p$-dimensional unit sphere. Then $w(K^+\bigcap\mathbb{S}^{p-1})\asymp\sqrt{p}$.
\end{lemma}

Suppose the eigengap $\nu$ doesn't scale with $n$. If $p<n$, we will have $\sqrt{\frac{w(K\bigcap\mathbb{S}^{p-1})}{(\nu\wedge\sqrt{\nu})\sqrt{n}}}>\frac{w(\mathcal{T}_K(\bx)\bigcap\mathbb{S}^{p-1})}{(\nu\wedge\sqrt{\nu})\sqrt{n}}$. Under the spiked covariance model, the upper bound for the idealized estimator in \eqref{ideal_upper1} and \eqref{ideal_upper2} reduce to
\begin{align*}
    \mathbb{E}\|\bar{v} - \bx\|\wedge\|\bar{v} + \bx\| \lesssim \frac{\sqrt{p}}{(\nu\wedge\sqrt{\nu})\sqrt{n}}
\end{align*}
The upper bound for the cone projected power iteration algorithm reduces to the same order if $\nu \geq (3 + 2\sqrt{2})\omek$ is satisfied. Then we investigate the lower bound. In the non-negative orthant $K^+$, all vector pairs have a non-negative dot product, so the finer set $K'$ in Corollary \ref{minimax_coro} could be $K^+$ itself. Plug in the fact $w(K^+\bigcap\mathbb{S}^{p-1})\asymp\sqrt{p}$ to Corollary \ref{minimax_coro} to get the following lower bound
\begin{align*}
    \inf_{\hat v} \sup_{\bx \in K^+ \bigcap \mathbb{S}^{p-1}} \mathbb{E} \|\hat v - \bx\|\wedge \|\hat v + \bx\| \gtrsim\frac{\sqrt{p}}{\log p(\nu\wedge\sqrt{\nu})\sqrt{n}}
\end{align*}
which differs up to a $\log p$ factor from the upper bound. Moreover, in a high-dimensional setting $p>n$, we might not be able to produce a consistent estimation of the principal eigenvector $\bx$ if $\bx\in K^+$ is the only information. This is because the ratio $\frac{p}{\log{p}\,n}$ tend to diverge for a large $p$.
This implies that the non-negative orthant $K^+$ is still too large to be an efficient constraint in high dimensions.

\subsubsection{Subspace}
\label{subsp:subsec}
Another example is an $s$-dimensional subspace $K=\mathbb{R}^s$. It is easy to show that an $s$-dimensional subspace is a convex cone. By rotating the data, without loss of generality we may  assume that $K$ is a set of vectors whose last $p-s$ coordinates are zero, i.e.,  $K=\{\mathbf{v}\in\mathbb{R}^p: v_j = 0,\,\forall j > s\}$. In this case $\mathcal{T}_K(\bx)=K$ for all $\bar x \in \mathbb{R}^p$. Accordingly we have $w(K\bigcap\mathbb{S}^{p-1}) = w(\mathcal{T}_K(\bx)\bigcap\mathbb{S}^{p-1}) = w(\mathbb{S}^{s-1})\asymp\sqrt{s}$ \citep[Example 7.5.7]{vershynin2018high}. Suppose the dimension of subspace $s$ is smaller than the sample size $n$. Thus we have $\sqrt{\frac{w(K\bigcap\mathbb{S}^{p-1})}{(\nu\wedge\sqrt{\nu})\sqrt{n}}} > \frac{w(\mathcal{T}_K(\bx)\bigcap\mathbb{S}^{p-1})}{(\nu\wedge\sqrt{\nu})\sqrt{n}}$. Under the spiked covariance model, the upper bound for the idealized estimator in \eqref{ideal_upper1}, \eqref{ideal_upper2} reduce to
\begin{align}\label{subspace:rate}
    \mathbb{E}\|\bar{v}-\bar{x}\|\wedge\|\bar{v} + \bx\| \lesssim \frac{\sqrt{s}}{(\nu\wedge\sqrt{\nu})\sqrt{n}}
\end{align}
The upper bound for the cone projected power iteration in \eqref{algo1_error_spiked} reduces to the same order if $\nu \geq (3 + 2\sqrt{2})\omek$ is satisfied. We can see that the cone projected power iteration algorithm gives a consistent estimation of the principal eigenvector $\bx$ as long as the right hand side of \eqref{subspace:rate} converges to $0$ as $n$ increases. For the lower bound, in Corollary \ref{minimax_coro}, $K'$ can be chosen as the $s$-dimensional non-negative orthant. By Lemma \ref{gw_nonneg:lemma} we have $w(K'\bigcap\mathbb{S}^{p-1})\asymp\sqrt{s}$, so that we are able to obtain a lower bound as 
\begin{align*}
    \inf_{\hat v} \sup_{\bx \in K \bigcap \mathbb{S}^{p-1}} \mathbb{E} \|\hat v - \bx\|\wedge \|\hat v + \bx\| \gtrsim\frac{\sqrt{s}}{\log p(\nu\wedge\sqrt{\nu})\sqrt{n}}
\end{align*}
which differs up to a $\log p$ factor from the upper bound.

\subsubsection{Monotone Cone}
\label{mnt_cone:subsec}
Consider the case where $M \subset \mathbb{R}^p$ is the monotone cone given by $M = \{(x_1, ..., x_p)\T\in\mathbb{R}^p\,: x_1\leq...\leq x_p \}$.
For a monotone cone $M$ there exist explicit formulas for the Gaussian complexity $w(M \bigcap \mathbb{S}^{p-1})\asymp\sqrt{\log p}$ for large $p$. This is proved in \citet[Section~D.4]{amelunxen2014living}, and also can be calculated numerically through Monte Carlo simulations \citep[Lemma~4.2]{donoho2013accurate}.
For a piecewise constant vector $\bx \in M$ with $m$ constant pieces, the order of tangent cone of $M$ at $\bar{x}$ has an explicit order as $w(\mathcal{T}_{M}(\bx)\bigcap\mathbb{S}^{p-1})\asymp\sqrt{m\log \frac{ep}{m}}$ \citep[Proposition~3.1]{bellec2018sharp}. Suppose now that the number of constant pieces $m$ doesn't scale with $n$, and $\log p<n$. Thus we have $\sqrt{\frac{w(K\bigcap\mathbb{S}^{p-1})}{(\nu\wedge\sqrt{\nu})\sqrt{n}}} > \frac{w(\mathcal{T}_K(\bx)\bigcap\mathbb{S}^{p-1})}{(\nu\wedge\sqrt{\nu})\sqrt{n}}$. Under the spiked covariance model, plug the Gaussian complexity $w(\mathcal{T}_{M}(\bx)\bigcap\mathbb{S}^{p-1})$ into \eqref{ideal_upper1}, \eqref{ideal_upper2} to get a reduced upper bound for the idealized estimator
\begin{align*}
    \mathbb{E}\|\bar{v} - \bx\|\wedge\|\bar{v} + \bx\| \lesssim
    \frac{\sqrt{\log p}}{(\nu\wedge\sqrt{\nu})\sqrt{n}}
\end{align*}
The upper bound for cone projected power iteration in \eqref{algo1_error_spiked} also reduces to the same order if  $\nu \geq (3 + 2\sqrt{2})\omek$ is satisfied. Our upper bound shows that the cone projected power iteration algorithm with monotone cone constraint is able to estimate the principal eigenvector consistently, as long as $\log p \sim o(n)$. We will now exhibit a lower bound for a piecewise constant $\bx$ setting. Similarly to Theorem \ref{minimax}, the proof of Proposition \ref{example_lower} relies on Fano's inequality. However the construction used to obtain a packing set is quite different. We note that the above upper bound above and the lower bound in Proposition \ref{example_lower} are different up to a $\log$ factor.
\begin{proposition}[Minimax Lower Bound for Monotone Cone]
\label{example_lower}
Suppose we have $n$ i.i.d. observations $X_1,...,X_n$ where $X_i\sim \mathcal{N}(0, I + \nu\bar{x}\bar{x}\T )$. Let $\bar{x} \in M$ where $M\subset\mathbb{R}^p$ is a monotone cone. The minimax risk of estimations of $\bx$ based on $X_1,...,X_n$ is bounded below by
\[
    \inf\limits_{\hat{v}}\max\limits_{\bx\in M} 
    \mathbb{E}[\|\hat{v} - \bx\|\wedge\|\hat{v} + \bx\|] 
    \gtrsim
    \frac{\sqrt{\log\log p}}{(\nu\wedge\sqrt{\nu})\sqrt{n}}
\]
\end{proposition}

\section{Experiments}
\label{experiment:section}
We compare our proposed cone projected power iteration algorithm with the ordinary power iteration \citep{mises1929praktische}, Truncated Power Iteration \citep{yuan2013truncated}, and ElasticNet SPCA \citep{zou2006sparse} on both simulated and real data sets. Truncated Power Iteration \citep{yuan2013truncated}, as the name suggests, is also a power method, but it truncates a certain proportion of coordinates to zero in every iteration to achieve sparsity. ElasticNet SPCA \citep{zou2006sparse} computes the non-sparse principal component first, and obtains the sparse eigenvector by performing elastic net regression of variables on the principal components. According to the simulation results, our algorithm has a smaller estimation error than the other three algorithms in both non-sparse and certain sparse settings of the principal eigenvector. The experiments on real data show that our algorithm provides the most consistent estimation of the principal eigenvector whose direction catches the largest proportion of variance in the test data. Code for experiments can be found in: \url{https://github.com/Pythongoras/ConeProjectedPowerIterV2}.

\subsection{Simulations}
In this section, we gauge the performance of Algorithm \ref{algo2} by comparing its estimation error and run time with three other algorithms -- ordinary power iteration, Truncated Power Iteration, and ElasticNet SPCA -- on simulated data. 

The data set is generated from a multivariate Gaussian distribution with zero mean and spiked covariance matrix $\Sigma = I + \nu\bx\bx\T$, where $\bx$ is the principal eigenvector. For the non-sparse setting, $\bx$ is formulated as $\bx_i = i/\sqrt{\sum_{j=1}^p j^2}$; for the sparse setting, $\bx$ is formulated as $\{\bx: \bx_i = 0 \text{ if }i \leq \lceil p-10\log p\rceil, \,\,\bx_i = 1 \text{ if }i > \lceil p-10\log p \rceil\}$, where $p$ is the dimension of $\bx$. One can observe that $\bx$ is monotonically increasing in both settings. We examine different values of $p$: $p=100,\,p=1000$, and $p=10000$. For each scale of $p$, we also examine different $n$ values: $10\log p$, $0.3p$, $p$, $5p$ and $10p$. Additionally, we test on both small ($\nu=0.5$) and large ($\nu=\log p)$ eigengaps to make sure the experiments have a reasonable coverage. 
For each combination of $\bx$, $p$, $n$ and $\nu$, the data matrix $X\in\mathbb{R}^{n\times p}$ is generated by drawing $n$ samples from $\mathcal{N}\big( 0, \Sigma \big)$, and all the four algorithms are implemented to estimate the principal eigenvector of the empirical covariance matrix $\hat{\Sigma} = \frac{1}{n}X\T X$. The convex cone used in Algorithm \ref{algo2} is the monotone cone. In this way, Algorithm \ref{algo2} is easy to implement since the projection onto the monotone cone can be done efficiently by isotonic regression.
The stopping criteria $\Delta = 10^{-6}$ is used throughout all power algorithms. The hyperparameters in ElasticNet SPCA and Truncated Power Iteration are tuned by grid search. The average $L_2$ distance of the estimated eigenvector to the true eigenvector $\bx$ for non-sparse $\bx$ case is shown in Fig.~\ref{fig:err_nonsparse}; for sparse $\bx$ case in Fig.~\ref{fig:err_sparse}. The average run time plots are attached in Appendix \ref{app_sim} for conciseness.

\begin{figure}[ht]
\vspace*{-10pt}
\resizebox{16cm}{!}{%
\includegraphics{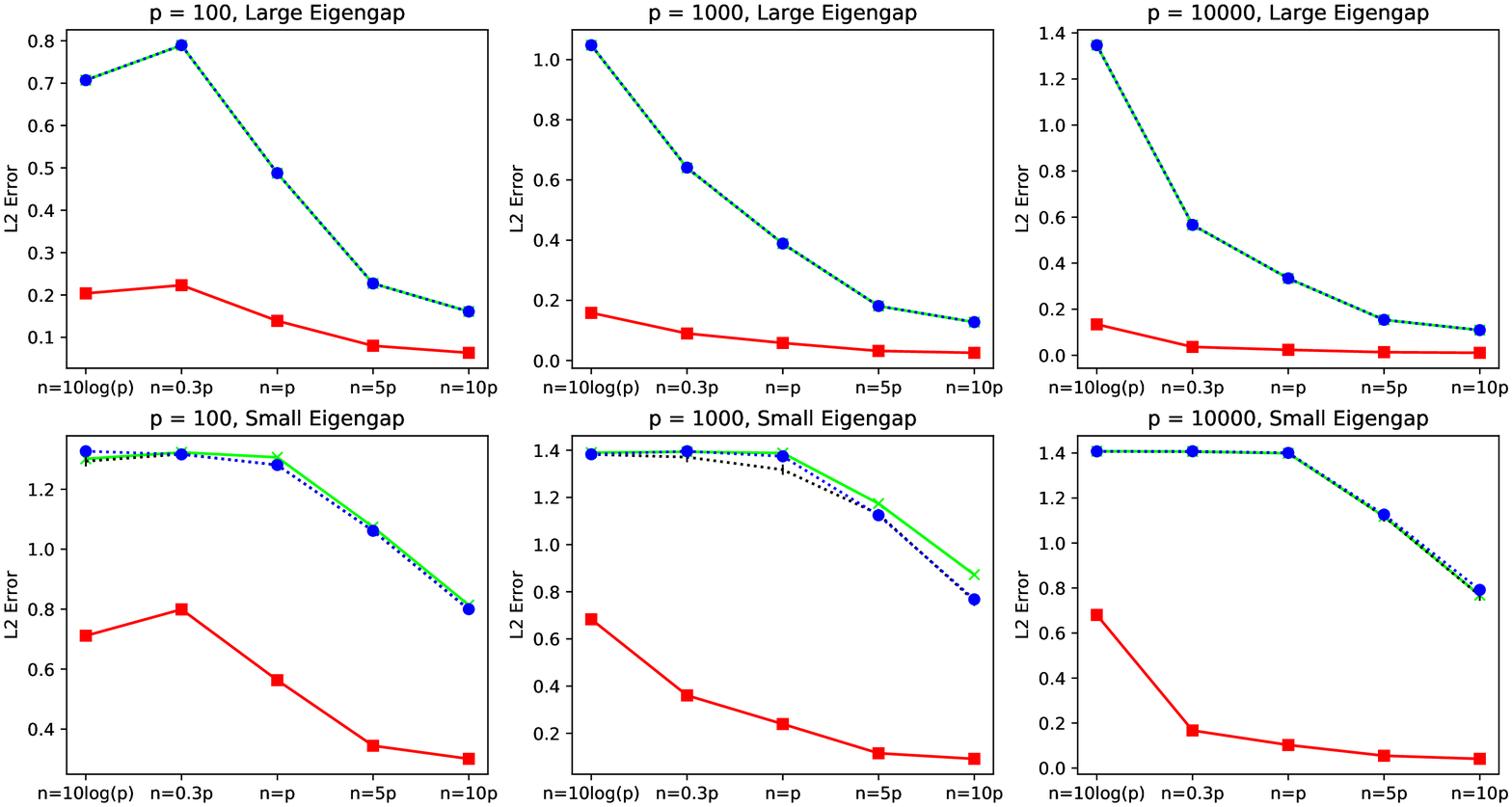}
}
\vspace*{-40pt}
\caption{ The $L_2$ Error for Different Algorithms on Simulated Matrix with Non-sparse principal eigenvector. \\
\protect\tikz{
\protect\draw[black,dotted,line width=0.7pt](0,0) -- (4mm,0);
\protect\draw[black,solid,line width=0.7pt](2mm,-0.1) -- (2mm,0.1)}, 
{\small ordinary power iteration}; 
\protect\tikz{
\protect\draw[green,solid,line width=0.5pt](0,0) -- (5mm,0);
\protect\draw[green,solid,line width=0.5pt](1.3mm,0.1) -- (2.7mm,-0.1);
\protect\draw[green,solid,line width=0.5pt](2.7mm,0.1) -- (1.3mm,-0.1)}, 
{\small Truncated Power Iteration}; 
\protect\tikz{
\protect\draw[blue,dotted,line width=0.7pt](0,0) -- (5mm,0);
\protect\filldraw[blue] (2mm,0) circle (2pt)}, 
{\small ElasticNet SPCA}; 
\protect\tikz{
\protect\draw[red,solid,line width=0.9pt](0,0) -- (5mm,0);
\protect\filldraw[red] (2mm,-0.05) rectangle ++(3pt,3pt)},
{\small Algorithm \ref{algo2}}.
}
\label{fig:err_nonsparse}
\vspace*{-12pt}
\end{figure}

\begin{figure}[ht]
\vspace*{-10pt}
\resizebox{16cm}{!}{%
\includegraphics{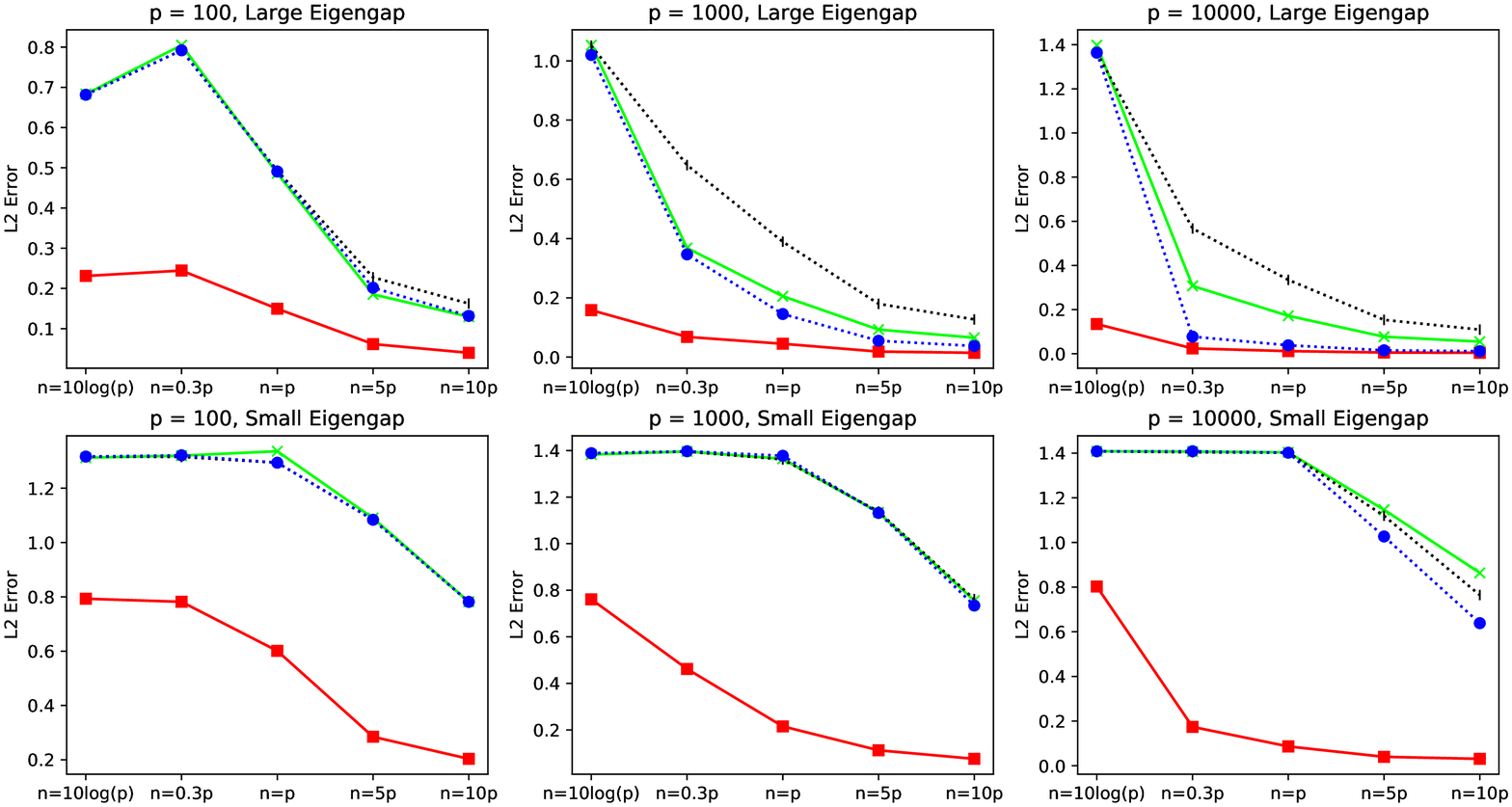}
}
\vspace*{-40pt}
\caption{The $L_2$ Error for Different Algorithms on Simulated Matrix with Sparse principal eigenvector.\\
\protect\tikz{
\protect\draw[black,dotted,line width=0.7pt](0,0) -- (4mm,0);
\protect\draw[black,solid,line width=0.7pt](2mm,-0.1) -- (2mm,0.1)}, 
{\small ordinary power iteration}; 
\protect\tikz{
\protect\draw[green,solid,line width=0.5pt](0,0) -- (5mm,0);
\protect\draw[green,solid,line width=0.5pt](1.3mm,0.1) -- (2.7mm,-0.1);
\protect\draw[green,solid,line width=0.5pt](2.7mm,0.1) -- (1.3mm,-0.1)}, 
{\small Truncated Power Iteration}; 
\protect\tikz{
\protect\draw[blue,dotted,line width=0.7pt](0,0) -- (5mm,0);
\protect\filldraw[blue] (2mm,0) circle (2pt)}, 
{\small ElasticNet SPCA}; 
\protect\tikz{
\protect\draw[red,solid,line width=0.9pt](0,0) -- (5mm,0);
\protect\filldraw[red] (2mm,-0.05) rectangle ++(3pt,3pt)},
{\small Algorithm \ref{algo2}}.
}
\label{fig:err_sparse}
\vspace*{-10pt}
\end{figure}

Based on Fig.~\ref{fig:err_nonsparse} \& \ref{fig:err_sparse}, our Algorithm \ref{algo2} provides a more precise estimation than other three algorithms for both non-sparse and sparse settings. Especially when $n \leq p$ or $\nu$ is small, the other three algorithms fail to converge, but Algorithm \ref{algo2} still achieves a small $L_2$ error. Of course in the non-sparse setting, Algorithm \ref{algo2} is expected to perform better than the sparse algorithms, which coincides with the experiment. Surprisingly, in the sparse setting, the sparse algorithms still converge slower than Algorithm \ref{algo2}. Our conjecture is that the monotone cone is very ``small'', so the projection step pushes the iterated vector towards the true eigenvector very fast. 
In addition, notice that the $L_2$ error of Truncated Power Iteration and ElasticNet SPCA are often very close to that of the ordinary power iteration. This is because the tuned $\lambda$ in ElasticNet SPCA is very small and the tuned cardinality in Truncated Power Iteration is very close to $1$. In such a regime, the two sparse algorithms are actually almost equivalent with the ordinary power iteration.

In terms of the run time, the actual magnitude of run time might not be of interests since they vary on different machines. Nevertheless, in Appendix \ref{app_sim}, the comparison over four algorithms on the same machine still gives out useful insights. Algorithm \ref{algo2} has more advantage when $p$ gets larger.
When $p=100$, Algorithm \ref{algo2} is slower than others, but it is still very fast given the small scale of all run times in this setting.
It is remarkable that the run time of Algorithm \ref{algo2} is even smaller than the ordinary power iteration when $p$ is large. This is due to less iterations Algorithm \ref{algo2} needs to converge, even though in every iteration it is more costly than the ordinary power iteration.

\subsection{Air Quality Data}
In this section we examine Algorithm \ref{algo2}, ordinary power iteration, Truncated Power Iteration and ElasticNet SPCA on an air pollution data. This data source is newly released by \citet{air_data} to support COVID-19 air quality based research. 

Our cleaned data set consists of historical observations of five air pollutant species— PM2.5, PM10, Ozone, CO, SO$_2$ —for several major cities in the world. The observations are collected for the first half of each year from 2015 to 2018 on a daily basis, but are not necessarily consecutive. 
The dimension of the sub-dataset for each air pollutant specie can be found in Table \ref{tab:air_dim}. According to the result of Augmented Dickey-Fuller test \citep[Section 2.7.5]{tsay2005analysis}, each air pollutant specie data set consists of stationary time series. An example time plot of PM2.5 observations in Tokyo is in Figure \ref{fig:tokyo_pm25_all}. 

\begin{table}[ht]
\begin{minipage}[b]{0.3\linewidth}
\centering
\resizebox{6cm}{!}{%
\includegraphics{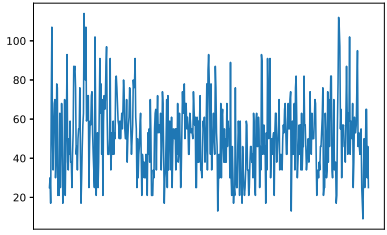}
}
\vspace*{-30pt}
\captionof{figure}{\footnotesize Time Plot of PM2.5 Observations in Tokyo}
\label{fig:tokyo_pm25_all}
\end{minipage}\hfill
\begin{minipage}[b]{0.6\linewidth}
\centering
\resizebox{9cm}{!}{%
\begin{tabular}{lccccc}
& Ozone & PM2.5 & PM10 & CO & SO2 \\[5pt]
number of cities ($n$) &  102 & 156 & 189 & 113 & 114 \\
number of observations ($p$) & 572 & 526 & 580 & 584 & 584
\end{tabular}
}
\vspace*{-6pt}
\captionof{table}{\footnotesize The Dimension of Data Set of Each Pollutant Specie}
\label{tab:air_dim}
\vskip 0.5cm
\resizebox{10cm}{!}{%
\includegraphics{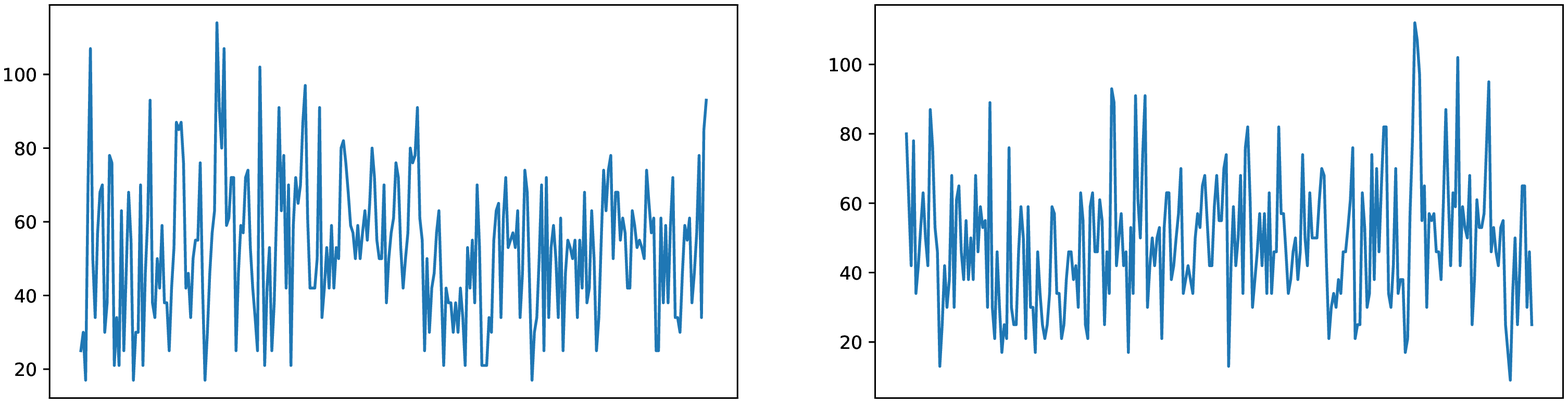}
}
\vspace*{-20pt}
\captionof{figure}{\footnotesize Time Plot of training set(left) and test set(right) of PM2.5 Observations in Tokyo.}
\label{fig:tokyo_pm25_split}
\end{minipage}
\end{table}




\begin{table}[ht]
\centering
\resizebox{15cm}{!}{%
\begin{tabular}{lcccc}
& Algorithm \ref{algo2} with Monotone Cone & Ordinary Power Iteration & Truncated Power Iteration & ElasticNet SPCA
 \\[5pt]
Ozone &  36.05\% & 33.44\% & 33.22\% & 33.43\%  \\
PM2.5 & 58.82\% & 56.62\% & 56.22\% & 56.62\% \\
PM10 & 50.64\% & 46.49\% & 46.02\% & 46.49\% \\
CO & 59.93\% & 54.77\% & 54.53\% & 54.73\% \\
SO2 & 58.64\% & 51.41\% & 51.40\% & 51.41\%
\end{tabular}
}
\caption{\small The Variance Explained by the First Principal Component for Different Algorithms for Air Quality Data}
\label{tab:air_var}
\end{table}

For each pollutant specie data set, we split the data set into two equal-length time series subsets: the train set and test set.
In this way, the train set contains observations of the first half 2015 and 2016; the test set contains observations of the first half 2017 and 2018. As is evident in Figure \ref{fig:tokyo_pm25_split} the train set and test set look similar. Moreover, since the data set consists of stationary time series, we expect the covariance matrices of the training set and test set to be similar. We apply Algorithm \ref{algo2}, ordinary power iteration, Truncated Power Iteration and ElasticNet SPCA on the covariance matrix of the train set to estimate its first principal eigenvector. Then we get four estimated principal eigenvectors corresponding to each algorithm, and construct principal components of test data using the four estimated principal eigenvectors respectively. The four algorithms are evaluated based on a standard criterion: the proportion of variance explained by the first principal component \citep{zou2006sparse}. Table \ref{tab:air_var} lists the results. Algorithm \ref{algo2} constantly beats the other three algorithms in terms of the variance explained by the first principal component.

\subsection{Dissolved Oxygen Data}
In this section we evaluate Algorithm \ref{algo2}, ordinary power iteration, Truncated Power Iteration and ElasticNet SPCA on a water quality data set which contains daily dissolved oxygen observations from 38 surface-water sites in the U.S. from 11/11/2016 to 11/11/2020. The raw data is obtained from \citet{water_data}. The data is not recorded consecutively, so there might be some dates when the observation is not available. The cleaned data set is of the size $38\times 744$. 

Similar to the previous section, in this section the data set is also split equally into the train and test set. The train set contains observations in 2016 and 2017, while the test set contains observations in 2018 and 2019. According to an example time plot of the Site \#2167716 in Figure \ref{fig:oxygen_all} and Figure \ref{fig:oxygen_split}, the data is stationary, and the train and test sets behave similarly. The results of Augmented Dickey-Fuller test \citep{tsay2005analysis} verifies the stationarity of data. Thus the covariance matrices of the train set and test set are expected to be similar. We apply Algorithm \ref{algo2}, ordinary power iteration, Truncated Power Iteration and ElasticNet SPCA to compute the first principal eigenvector of the train covariance matrix. Then four first principal components are constructed on test set using the four estimated principal eigenvectors. We evaluate the four algorithms using the proportion of variance explained by the first principal component \citep{zou2006sparse}. Results can be found in Table \ref{tab:water_var}. 

\begin{figure}[ht]
\begin{minipage}[b]{0.3\linewidth}
\centering
\resizebox{6cm}{!}{%
\includegraphics{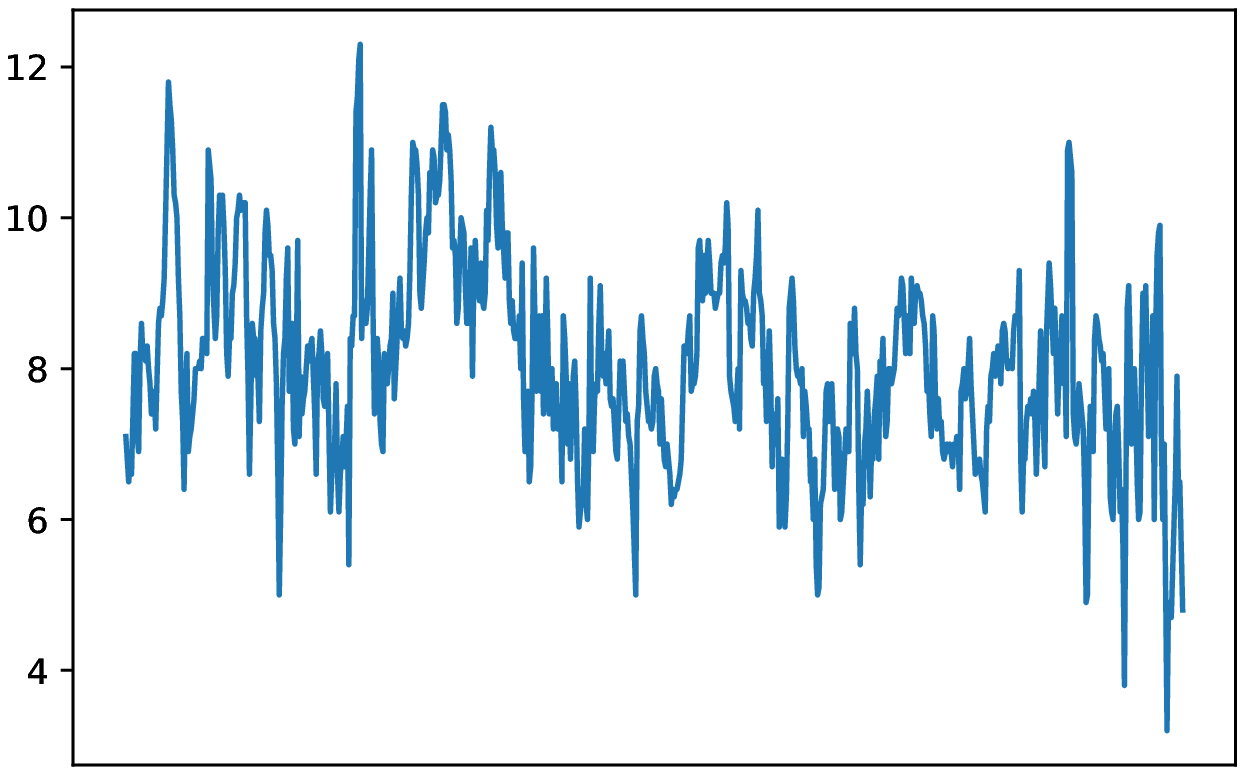}
}
\vspace*{-30pt}
\captionof{figure}{\footnotesize Time Plot of Dissolved Oxygen in Site \#2167716}
\label{fig:oxygen_all}
\end{minipage}\hfill
\begin{minipage}[b]{0.6\linewidth}
\centering
\vspace*{-10pt}
\resizebox{10cm}{!}{%
\includegraphics{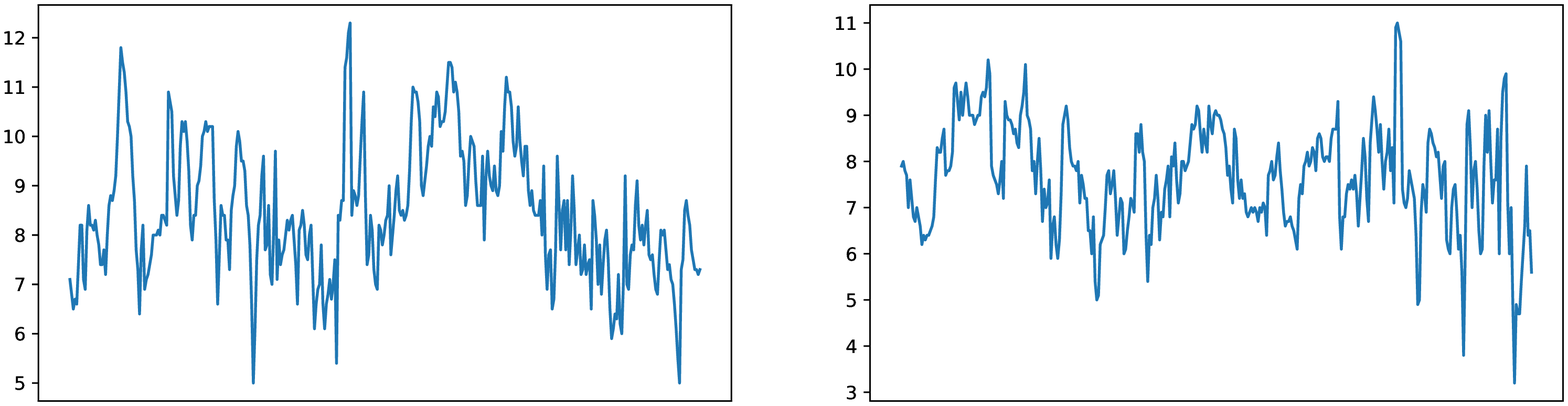}
}
\vspace*{-20pt}
\captionof{figure}{\footnotesize Time Plot of training set(left) and test set(right) of Dissolved Oxygen in Site \#2167716}
\label{fig:oxygen_split}
\end{minipage}
\end{figure}

\begin{table}[ht]
\centering
\resizebox{15.5cm}{!}{%
\begin{tabular}{lcccc}
& Algorithm \ref{algo2} with Monotone Cone & Ordinary Power Iteration & Truncated Power Iteration & ElasticNet SPCA
 \\[5pt]
Explained Variance & 65.74\% & 64.51\% & 64.19\% & 64.26\%
\end{tabular}
}
\caption{\small The Variance Explained by the First Principal Component for Different Algorithms for Dissolved Oxygen Data}
\label{tab:water_var}
\end{table}

\section{Discussions} \label{discussion:sec}
In this paper we propose a cone projected power iteration method to tackle the problem of finding the first principal eigenvector in a positive semidefinite matrix obscured by stochastic noise. Unlike some other high-dimensional PCA methods which require tuning parameters such as sparse PCA, our algorithm is hyperparameter free. To our knowledge, this paper is one of the first to give the time complexity analysis of a cone projected power iteration algorithm. Also it is one of the first to derive an error rate upper bound in a setting of positive semidefinite input matrix, which is much more general than the common spiked covariance model. 

We would like to point out three questions which are worth future investigations. The first is to close the gap between the upper and lower bound of the estimation error. As suggested by the monotone cone example in Section \ref{mnt_cone:subsec}, the packing set of the convex cone might be constructed more efficiently to produce a tighter lower bound. Such an efficient packing set might be a local packing of $\mathbb{B}_{\epsilon}(\mathbf{v}_0)\cap K$ at some point $\mathbf{v}_0\in K$ \citep{cai2020optimal}.
Another question is whether it is possible to extend the cone projected power iteration to estimate more than one principal eigenvectors. If the convex cone constraint $K$ is large enough to contain multiple orthogonal vectors, multiple eigenvectors might be estimated by iteratively applying the cone projected power iteration algorithm. The statistical analysis of the multi-eigenvector case is not trivial. 
%






\newpage

\appendix

\section{Additional Simulation Results}
\label{app_sim}
\begin{figure}[H]
\resizebox{16cm}{!}{%
\includegraphics{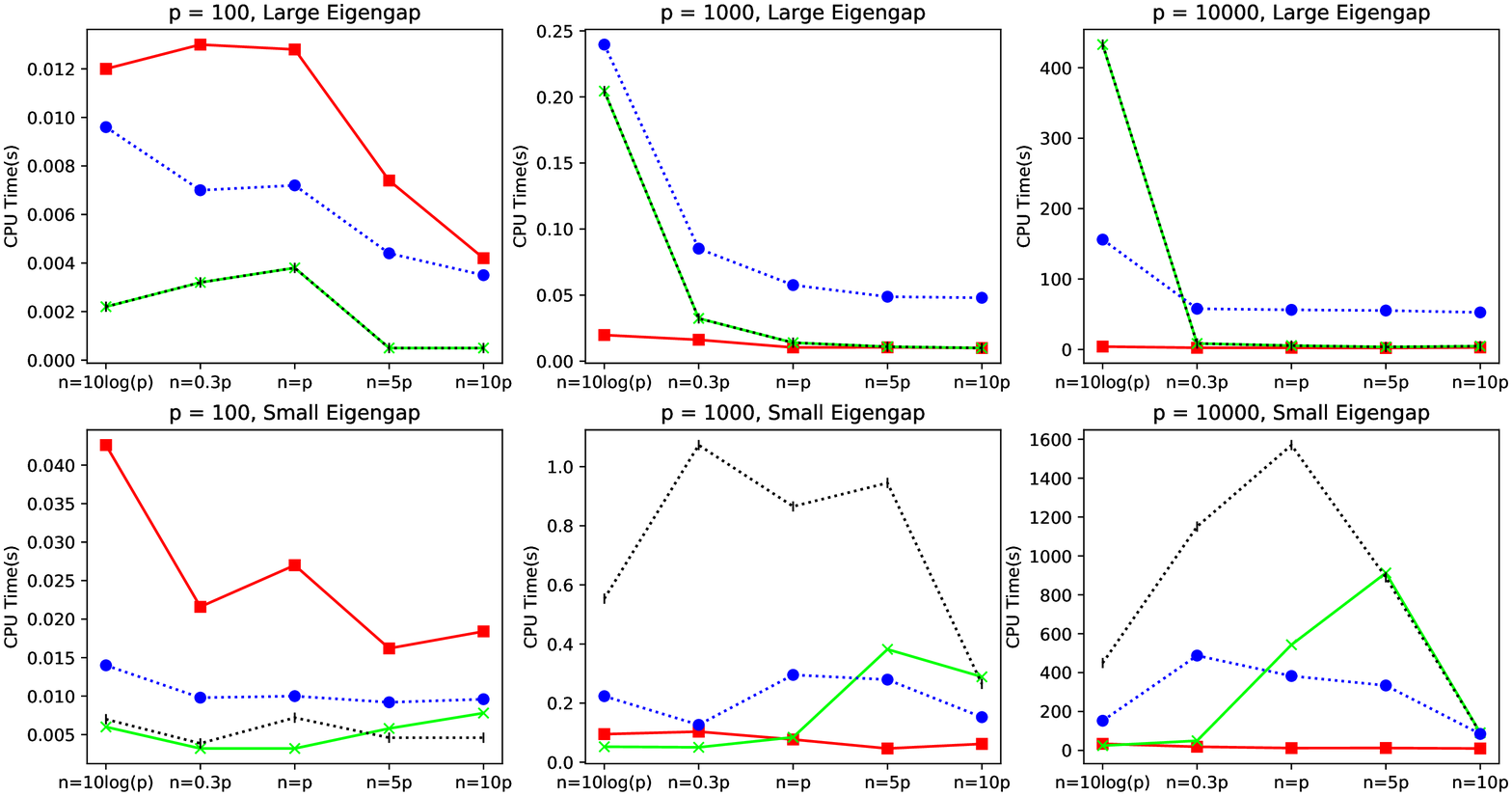}
}
\vspace*{-40pt}
\caption{The Run Time for Different Algorithms on Simulated Matrix with Non-sparse principal eigenvector. \\
\protect\tikz{
\protect\draw[black,dotted,line width=0.7pt](0,0) -- (4mm,0);
\protect\draw[black,solid,line width=0.7pt](2mm,-0.1) -- (2mm,0.1)}, 
{\small ordinary power iteration}; 
\protect\tikz{
\protect\draw[green,solid,line width=0.5pt](0,0) -- (5mm,0);
\protect\draw[green,solid,line width=0.5pt](1.3mm,0.1) -- (2.7mm,-0.1);
\protect\draw[green,solid,line width=0.5pt](2.7mm,0.1) -- (1.3mm,-0.1)}, 
{\small Truncated Power Iteration}; 
\protect\tikz{
\protect\draw[blue,dotted,line width=0.7pt](0,0) -- (5mm,0);
\protect\filldraw[blue] (2mm,0) circle (2pt)}, 
{\small ElasticNet SPCA}; 
\protect\tikz{
\protect\draw[red,solid,line width=0.9pt](0,0) -- (5mm,0);
\protect\filldraw[red] (2mm,-0.05) rectangle ++(3pt,3pt)},
{\small Algorithm \ref{algo2}}.
}
\label{fig:sim1}
\vspace*{-12pt}
\end{figure}

\begin{figure}[H]
\resizebox{16cm}{!}{%
\includegraphics{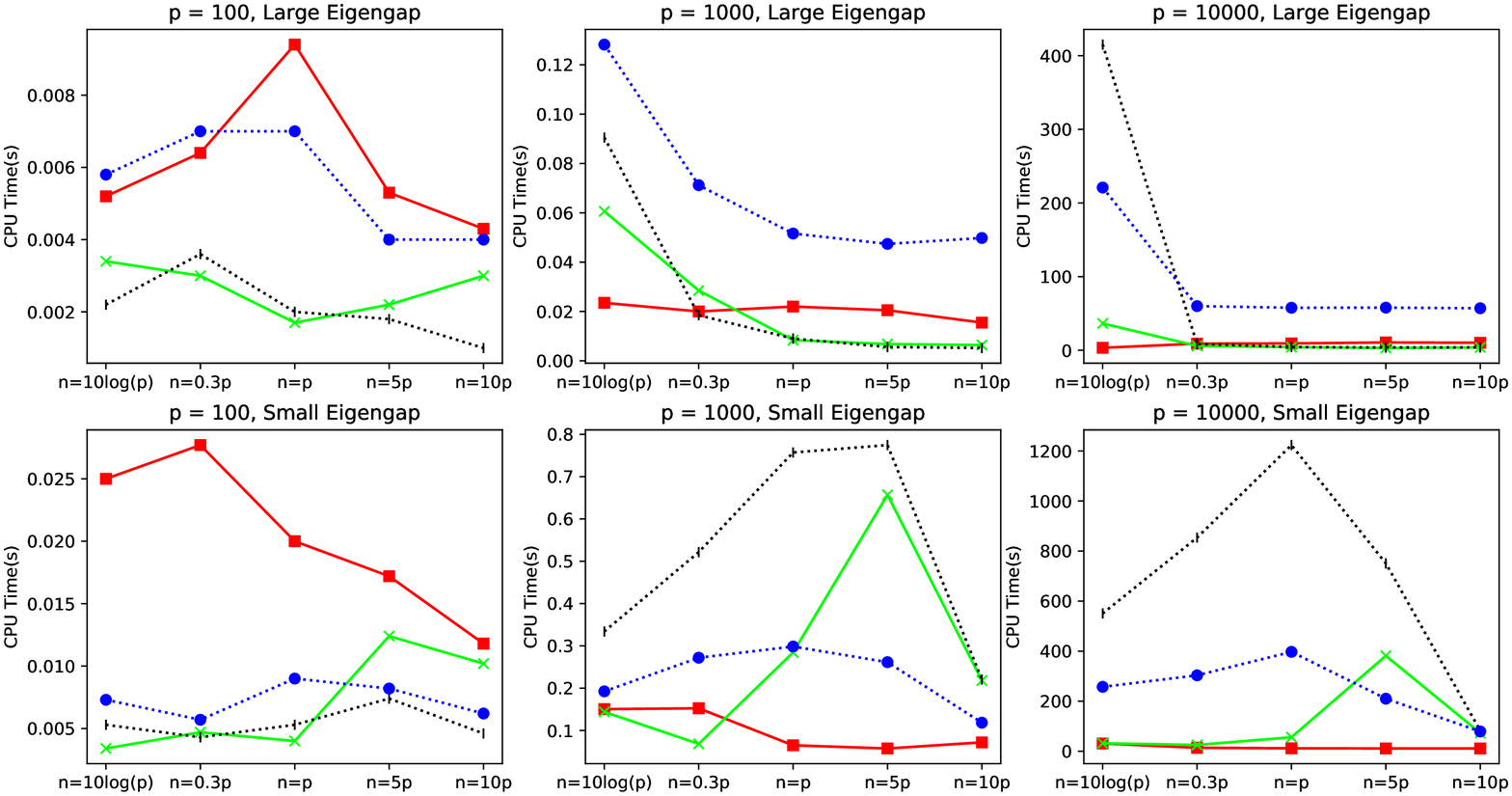}
}
\vspace*{-40pt}
\caption{The Run Time for Different Algorithms on Simulated Matrix with Sparse principal eigenvector.\\
\protect\tikz{
\protect\draw[black,dotted,line width=0.7pt](0,0) -- (4mm,0);
\protect\draw[black,solid,line width=0.7pt](2mm,-0.1) -- (2mm,0.1)}, 
{\small ordinary power iteration}; 
\protect\tikz{
\protect\draw[green,solid,line width=0.5pt](0,0) -- (5mm,0);
\protect\draw[green,solid,line width=0.5pt](1.3mm,0.1) -- (2.7mm,-0.1);
\protect\draw[green,solid,line width=0.5pt](2.7mm,0.1) -- (1.3mm,-0.1)}, 
{\small Truncated Power Iteration}; 
\protect\tikz{
\protect\draw[blue,dotted,line width=0.7pt](0,0) -- (5mm,0);
\protect\filldraw[blue] (2mm,0) circle (2pt)}, 
{\small ElasticNet SPCA}; 
\protect\tikz{
\protect\draw[red,solid,line width=0.9pt](0,0) -- (5mm,0);
\protect\filldraw[red] (2mm,-0.05) rectangle ++(3pt,3pt)},
{\small Algorithm \ref{algo2}}.
}
\label{fig:sim2}
\end{figure}

\begin{figure}[H]
\resizebox{16cm}{!}{%
\includegraphics{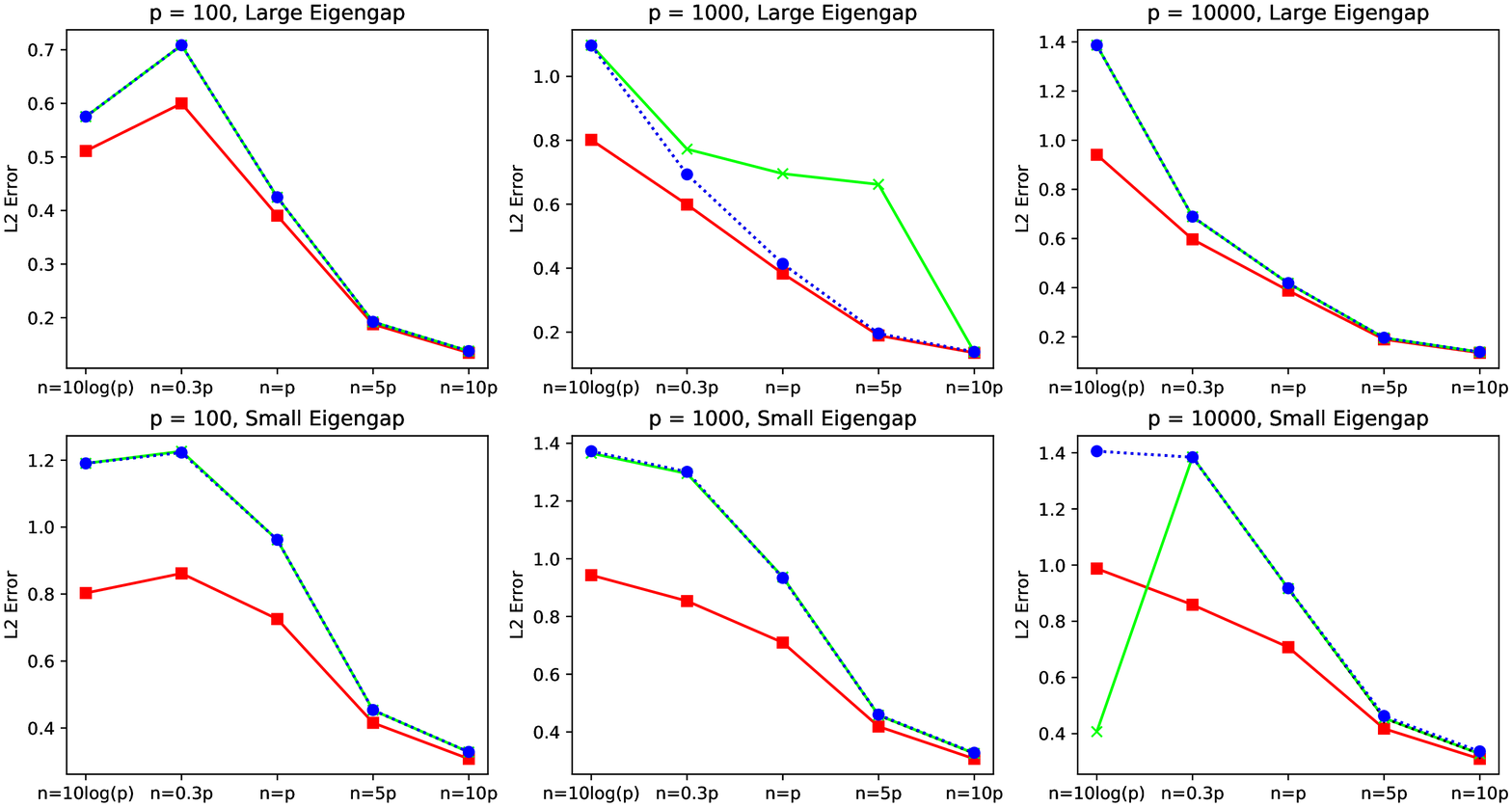}
}
\vspace*{-40pt}
\caption{The $L_2$ Error for Different Algorithms on Simulated Matrix with Sparse principal eigenvector.\\
\protect\tikz{
\protect\draw[black,dotted,line width=0.7pt](0,0) -- (4mm,0);
\protect\draw[black,solid,line width=0.7pt](2mm,-0.1) -- (2mm,0.1)}, 
{\small Ordinary Power Iteration}; 
\protect\tikz{
\protect\draw[green,solid,line width=0.5pt](0,0) -- (5mm,0);
\protect\draw[green,solid,line width=0.5pt](1.3mm,0.1) -- (2.7mm,-0.1);
\protect\draw[green,solid,line width=0.5pt](2.7mm,0.1) -- (1.3mm,-0.1)}, 
{\small Truncated Power Iteration}; 
\protect\tikz{
\protect\draw[blue,dotted,line width=0.7pt](0,0) -- (5mm,0);
\protect\filldraw[blue] (2mm,0) circle (2pt)}, 
{\small ElasticNet SPCA}; 
\protect\tikz{
\protect\draw[red,solid,line width=0.9pt](0,0) -- (5mm,0);
\protect\filldraw[red] (2mm,-0.05) rectangle ++(3pt,3pt)},
{\small Algorithm \ref{algo2} with Positive Cone}.
}
\label{fig:sim3}
\end{figure}

\section{Preliminary Results}
\label{app:theorem}

\subsection{Constants}\quad 
Below we define precisely the constants $c_{-1}$ and $c_1$. The intuition of defining such constants can be found in the proof of Lemma \ref{bigger_c0}. Let
\begin{align*}
    c_{\eta} & = \cc, \mbox{ for } \eta \in \{-1,1\}. 
\end{align*}
We have the following simple bounds for $c_{-1}$ and $c_1$. 
\begin{lemma}[Bounds for $c_{\pm 1}$]
\label{lemma:bounds:c1:cm1} 
Assuming that $\nu \geq (3 + 2\sqrt{2}) \|E\|_K$ we have the following relationships
\begin{align*}
    \frac{\|E\|_K}{\lambda} \leq c_{-1} \leq \frac{(3 + 2 \sqrt{2})\|E\|_K}{(1 +  \sqrt{2})\nu} \,\,\,\, \mbox{ and } \quad \frac{2}{5} < \frac{1}{1 + \sqrt{2}} \leq c_1. 
\end{align*}
\end{lemma}
\begin{proof}[Proof of Lemma \ref{lemma:bounds:c1:cm1}]
Note that $c_{-1},c_1$ are well defined when $\nu \geq (3 + 2\sqrt{2})\|E\|_K$. Furthermore the following holds:
\begin{align*}
\frac{\|E\|_K}{\lambda} \leq \frac{\|E\|_K}{\nu-\|E\|_K} & \leq \frac{2 \|E\|_K}{\nu-\|E\|_K + \sqrt{(\nu-\|E\|_K)^2 - 4 \|E\|_Kk}} = c_{-1},
\\ & \leq \frac{2\|E\|_K}{\nu - \|E\|_K} \leq \frac{(3 + 2 \sqrt{2})\|E\|_K}{(1 +  \sqrt{2})\nu}, 
\end{align*}
In addition 
\begin{align*}
    c_{1} \geq \frac{k - \omek}{2\nu} > \frac{1 + \sqrt{2}}{3 + 2\sqrt{2}} = \frac{1}{1 + \sqrt{2}}.
\end{align*}
\end{proof}

\begin{lemma}\label{simple:v0:lemma} Under the assumption of Lemma \ref{lemma:bounds:c1:cm1}, if $v_0\T \bar x > c_{-1}$ it follows that $v_0\T A v_0 > 0$.
\end{lemma}

\begin{proof}[Proof of Lemma \ref{simple:v0:lemma}] Consider $\bar x\T A v_0 \geq \lambda v_0\T \bar x - \|E\|_K > \lambda c_{-1} - \|E\|_K > 0$ by Lemma \ref{lemma:bounds:c1:cm1}. Hence by Cauchy-Schwartz $v_0\T A v_0 > 0$.
\end{proof}

\subsection{Useful Tools}\quad 
The following theorem states that any vector $z\in\mathbb{R}^p$ can be decomposed as the sum of its projection onto $K$, and its projection onto $K$'s polar cone $K^\circ = \{u: \langle u,v \rangle \leq 0, \, \forall v\in K\}$. It is a fundamental result in convex analysis. 
\begin{theorem}[
\citep{moreau1962decomposition}
(Moreau's Decomposition)]\label{moreausthm}
Let $K\subset\mathbb{R}^p$ be a convex cone and $K^o$ be its polar cone. For $x, y, z \in\mathbb{R}^p$, the following properties are equivalent:
\begin{enumerate}
    \item $z = x + y$, $x\in K$, $y\in K^\circ$, and $\langle x, y \rangle = 0$. 
    \item $x=\Pi_K z$, $y=\Pi_{K^\circ}z$.
\end{enumerate}
\end{theorem}

The following Lemma provides some equations and inequalities which will be referred a lot in consequent sections.
\begin{lemma}\label{ineqs} Suppose that the vectors $v_t$ are recursively defined as $v_t =
\frac{\Pi_KAv_{t-1}}
{\|\Pi_KAv_{t-1}\|}$.
\begin{enumerate}
    \item $\|\Pi_KAv_{t-1}\| = v_t\T Av_{t-1}$.
    \item $\forall x\in\mathbb{R}^p$, $x\T \bar{A}x \leq (\bar{x}\T x)^2\lambda + [1-(\bar{x}\T x)^2]\mu$.
    \item $\forall v \in K$ we have $|v\T Ev - \bx\T E\bx| \leq 4\|v-\bx\|\omtan$ and $|v\T Ev - \bx\T E\bx| \leq 4\|v+\bx\|\omek$.
    \item $v_t\T Av_{t-1} \leq v_t\T Av_t$.
\end{enumerate}
\end{lemma}
\begin{proof}[Proof of Lemma \ref{ineqs}]\quad 
\begin{enumerate}
    \item 
    \begin{align*}
        v_t\T Av_{t-1} & = \frac{\langle 
        \Pi_KAv_{t-1}, Av_{t-1}
        \rangle}{\|\Pi_KAv_{t-1}\|}
        = 
        \frac{\langle 
        \Pi_KAv_{t-1}, \Pi_KAv_{t-1}
        \rangle + 
        \langle 
        \Pi_KAv_{t-1}, \Pi_{K^\circ}Av_{t-1}
        \rangle
        }{\|\Pi_KAv_{t-1}\|}\\
        & = \|\Pi_KAv_{t-1}\|
        \qquad \qquad 
        \text{by Moreau's Decomposition}
        \end{align*}
    \item 
    One can write $x$ as $x = (\bx\T x)\bx + \sqrt{1-(\bx\T x)^2}\bx^{\perp}$. Then
    \begin{align*}
    x\T\bar{A}x = (\bx\T x)^2\lambda + 
    [1-(\bx\T x)^2]\bx^{\perp\intercal}\bar{A}\bx^{\perp} \leq
    (\bar{x}\T x)^2\lambda + [1-(\bar{x}\T x)^2]\mu
    \end{align*}
    \item On one hand
    \begin{align*}
        |v\T Ev - \bx\T E\bx| & = |(v-\bx)\T E(v-\bx)+2\bx\T E(v-\bx)|\\
        & \leq \|v-\bx\|^2 \omtan + 2\|v-\bx\|\omtan\\
        & \leq 4\|v-\bx\|\omtan.
    \end{align*}
    On the other hand, 
    \begin{align*}
    |v\T Ev - \bx\T  E\bx| \leq |2\bx\T  E(v + \bx)| + |(v + \bx)\T  E(v + \bx)|.
\end{align*}
Since $v + \bx \in K$ it follows that $\frac{v + \bx}{\|v + \bx\|} \in K$ and therefore
\begin{align*}
    |v\T Ev - \bx\T  E\bx| \leq 2\|v + \bx\|\|E\|_K + \|v + \bx\|^2 \|E\|_K \leq 4\|v + \bx\|\|E\|_K,
\end{align*}
    
\item By Cauchy-Schwartz inequality we have
\begin{align*}
    v_t\T Av_{t-1} \leq \sqrt{(v_t\T Av_t)(v_{t-1}\T A v_{t-1})} \leq v_t\T Av_t,
\end{align*}
where $v_t\T Av_t\geq v_{t-1}\T A v_{t-1}$ comes from the first part of the proof of Proposition 1.
\end{enumerate}
\end{proof}

\section{Proof of Lemma \ref{tan_gw}}

There are two definitions we need to introduce shortly. The first is the Orlicz norm $\|\cdot\|_{\psi_\alpha}$ of a random variable for $1\leq\alpha\leq2$
\begin{align*}
    \|X\|_{\psi_\alpha} = \inf\{ C > 0: \, \mathbb{E}\exp(|X/C|^{\alpha}) \leq 2 \}
\end{align*}
For a function $f$ defined on the probability space, Having a finite Orlicz norm is equivalent with displaying a tail behavior dominated by an exponential tail bound. If $\|X\|_{\psi_\alpha}<\infty$, one can show that $\mathbb{P}(|X|>t) \leq 2\exp(-ct^{\alpha}/\|X\|_{\psi_\alpha}^{\alpha})$ for $c>0, t\geq 1$. On the other hand, given $\mathbb{P}(|X|>t) \leq 2\exp(-t^{\alpha}/C^{\alpha})$, one can show that $\|X\|_{\psi_\alpha}\leq c_1C$ for $c_1>0, C>0$. We say that $X$ is a sub-Exponential random variable if $\|X\|_{\psi_1}<\infty$; a sub-Gaussian random variable if $\|X\|_{\psi_2}<\infty$.

The second definition is the Talagrand's $\gamma_2$ functional $\gamma_2(T,d)$, which measures the complexity of set $T$ with respect to metric $d$. The quantity $\gamma_2(T,d)$ provides a link to the Gaussian Complexity of $T$. The complete definition of $\gamma_2(T,d)$ can be found in \citet[Definition 2.4]{mendelson2010empirical}. Here we only need to know two properties of it
\begin{lemma}
\label{tala_lemmas}
\begin{enumerate}
    \item For some constant $k, C$, if $f: (T,d)\rightarrow(U, d')$ is onto and satisfies $d'(f(x), f(y)) \leq C d(x,y), \,\,\forall x,y\in T$, then $\gamma_{\alpha}(U, d') \leq kC \gamma_{\alpha}(T, d)$ \citep[Theorem 1.3.6]{talagrand2006generic}.
    \item Consider a Gaussian process $\{X_t\}_{t\in T}$ on $T$, and the canonical distance $d(s,t)=\sqrt{\mathbb{E}(X_t-X_s)^2}$. For some universal constant $c$ we have $\gamma_2(T,d)\leq c\, \mathbb{E} \sup_{t\in T}X_t$ \citep[Theorem 2.1.1]{talagrand2006generic}.
\end{enumerate}
\end{lemma}

Let start the proof of Lemma \ref{tan_gw}. The proof is based on Lemma \ref{emp_upper_bound}, which is an upper bound of the empirical process $\mathbb{E}\sup_{f\in\mathcal{F}}\big|\frac{1}{n}\sum_{i=1}^n f^2(Z_i)-\mathbb{E}f^2(Z_i) \big|$. See \citet[Lemma A.3.1]{vu2012minimax} and \citet[Theorem A]{mendelson2010empirical}.

\begin{lemma}
\label{emp_upper_bound}
Let $Z_i, i=1,...,n$ be i.i.d. random variables. There exists an absolute constant $c$ for which the following holds. If $\mathcal{F}$ is a symmetric class of mean-zero functions then
\begin{align*}
    \mathbb{E}\sup_{f\in\mathcal{F}}\Big|\frac{1}{n}\sum_{i=1}^n f^2(Z_i)-\mathbb{E}f^2(Z_i) \Big| \leq c \big[d_{\psi_1}\frac{\gamma_2(\mathcal{F},\psi_2)}{\sqrt{n}} \vee\frac{\gamma_2^2(\mathcal{F},\psi_2)}{n}\big]
\end{align*}
where $d_{\psi_1}=\sup_{f\in\mathcal{F}}\|f\|_{\psi_1}$.
\end{lemma}

\textbf{1). Decompose $\omek$}\\
To apply Lemma \ref{emp_upper_bound}, we start with decomposing $\omek$. Under the spiked covariance model, we have $X_i \sim \mathcal{N}(0, I + \nu \bar{x}\bar{x}\T)$. Let $\mathbb{X} = [X_1\T; X_2\T;...;X_n\T]$. With the fact $\sqrt{I+\nu\bx\bx\T}=I+(\sqrt{1+\nu}-1)\bx\bx\T$, we can rewrite the noise matrix $E$ as
\begin{align*}
E & = A - \bar{A}\\
& = \frac{\mathbb{X}\T \mathbb{X}}{n} - (I + \nu \bar{x}\bar{x}\T)\\
& = \big(I+(\sqrt{\nu+1}-1)\bx\bx\T \big)
\big(\frac{\tilde{X}\T\tilde{X}}{n}-I\big)
\big(I+(\sqrt{\nu+1}-1)\bx\bx\T \big)
\end{align*}
where $\tilde{X}$ is a standard Gaussian matrix. Let $\nu_0 = \sqrt{\nu+1}-1$, then
\begin{align}
\omek & = \sup\limits_{\|x\|=\|y\|=1,\,\,x,y\in K} \Big|x\T (I+\nu_0\bx\bx\T )(\frac{\tilde{X}\T\tilde{X}}{n}-I)(I+\nu_0\bx\bx\T )y \Big|\nonumber\\
& = \sup\limits_{\|x\|=\|y\|=1,\,\,x,y\in K} \Big|[x+\nu_0(\bx\T x)\bx]\T (\frac{\tilde{X}\T\tilde{X}}{n} - I)[y+\nu_0(\bx\T y)\bx] \Big|\nonumber\\
& \leq \sup\limits_{\|x\|=\|y\|=1,\,\,x,y\in K}|x\T (\frac{\tilde{X}\T\tilde{X}}{n} - I)y| + \nu_0\sup\limits_{\|y\|=1,\,\,y\in K}|\bx\T(\frac{\tilde{X}\T\tilde{X}}{n} - I)y|\nonumber \\
& + \nu_0\sup\limits_{\|x\|=1,\,\,x\in K}|\bx\T (\frac{\tilde{X}\T\tilde{X}}{n} - I)x| + \nu_0^2|\bx\T (\frac{\tilde{X}\T\tilde{X}}{n} - I)\bx|
\label{omek_decomp}
\end{align}
With an algebraic trick one can show
\begin{align*}
    & \sup\limits_{\|x\|=\|y\|=1,\,\,x,y\in K}|x\T (\frac{\tilde{X}\T\tilde{X}}{n} - I)y| \\
    = & \sup\limits_{\|x\|=\|y\|=1,\,\,x,y\in K}\frac{1}{2}|(x+y)\T(\frac{\tilde{X}\T\tilde{X}}{n} - I)(x+y) - x\T(\frac{\tilde{X}\T\tilde{X}}{n} - I)x-y\T(\frac{\tilde{X}\T\tilde{X}}{n} - I)y|\\
    \leq & \sup\limits_{\|x\|=\|y\|=1,\,\,x,y\in K}\frac{1}{2}|(x+y)\T(\frac{\tilde{X}\T\tilde{X}}{n} - I)(x+y)| + \sup\limits_{\|x\|=1,\,\,x\in K}\frac{1}{2}|x\T(\frac{\tilde{X}\T\tilde{X}}{n} - I)x|\\
    & + \sup\limits_{\|y\|=1,\,\,y\in K}\frac{1}{2}|y\T(\frac{\tilde{X}\T\tilde{X}}{n} - I)y|
\end{align*}
Let $z=x+y$. Since $K$ is a convex cone, $z$ is also in $K$. We have
\begin{align*}
    & \sup\limits_{\|x\|=\|y\|=1,\,\,x,y\in K}\frac{1}{2}|(x+y)\T(\frac{\tilde{X}\T\tilde{X}}{n} - I)(x+y)|\\
    \leq & \sup\limits_{\|z\|=2,\,\,z\in K}\frac{1}{2}|z\T(\frac{\tilde{X}\T\tilde{X}}{n} - I)z| = \sup\limits_{\|z\|=1,\,\,z\in K}2|z\T(\frac{\tilde{X}\T\tilde{X}}{n} - I)z|
\end{align*}
so that 
\begin{align*}
    \sup\limits_{\|x\|=\|y\|=1,\,\,x,y\in K}|x\T (\frac{\tilde{X}\T\tilde{X}}{n} - I)y| \leq 3\sup\limits_{\|x\|=1,\,\,x\in K}|x\T(\frac{\tilde{X}\T\tilde{X}}{n} - I)x|
\end{align*}
The same trick can be applied to the other two terms $\sup\limits_{\|y\|=1,\,\,y\in K}|\bx\T(\frac{\tilde{X}\T\tilde{X}}{n} - I)y|$ and $\sup\limits_{\|x\|=1,\,\,x\in K}|\bx\T (\frac{\tilde{X}\T\tilde{X}}{n} - I)x|$ in \eqref{omek_decomp}:
\begin{align*}
    \sup\limits_{\|x\|=1,\,\,x\in K}|\bx\T(\frac{\tilde{X}\T\tilde{X}}{n} - I)x| & \leq \frac{5}{2}\sup\limits_{\|x\|=1,\,\,x\in K}|x\T(\frac{\tilde{X}\T\tilde{X}}{n} - I)x| + \frac{1}{2}|\bx\T (\frac{\tilde{X}\T\tilde{X}}{n} - I)\bx|
\end{align*}
so we can get
\begin{align*}
    \omek & \leq (5\nu_0+3) \sup\limits_{\|x\|=1,\,\,x\in K}|x\T(\frac{\tilde{X}\T\tilde{X}}{n} - I)x| + (\nu_0^2+\nu_0)|\bx\T (\frac{\tilde{X}\T\tilde{X}}{n} - I)\bx|
\end{align*}

\textbf{2). Bound $\mathbb{E}|\bx\T (\frac{\tilde{X}\T\tilde{X}}{n} - I)\bx|$}.\\
Notice that $\bx\T (\frac{\tilde{X}\T\tilde{X}}{n} - I)\bx = \frac{1}{n}\sum\limits_{i=1}^n [(\tilde{X}\bx)^2_i - 1]$, and $(\tilde{X}\bx)^2_i\sim\chi_1^2$. Let $Z=\frac{1}{n}\sum\limits_{i=1}^n [(Xv)^2_i - 1]$. The following concentration inequalities hold for $Z$ \citep[Lemma 1]{laurent2000adaptive},
    \begin{align*}
        \mathbb{P}\big( 
        Z & \geq 
        2\sqrt{\frac{x}{n}} + 2\frac{x}{n} 
        \big) \leq e^{-x} \\
        \mathbb{P}\big( 
        Z & \leq 
        -2\sqrt{\frac{x}{n}} 
        \big) \leq e^{-x}
    \end{align*}
so that
\begin{align*}
    \mathbb{P}\big(\sqrt{n}Z\geq t\big) & \leq e^{-\frac{t^2}{16}}\qquad\text{if }t<4\sqrt{n}\\
    \mathbb{P}\big(\sqrt{n}Z\geq t\big) & \leq e^{-\frac{\sqrt{n}t}{4}}\qquad\text{if }t\geq4\sqrt{n}\\
    \mathbb{P}\big(\sqrt{n}Z\leq -t\big) & \leq e^{-\frac{t^2}{4}}
\end{align*}
The expectation can be bounded as
\begin{align*}
    \mathbb{E}|\sqrt{n}Z| & = \int_0^{\infty} \mathbb{P}(\sqrt{n}|Z|>t)dt\\
    & = \int_{t<4\sqrt{n}}e^{-\frac{t^2}{16}}dt + \int_{t>4\sqrt{n}}e^{-\frac{\sqrt{n}t}{4}}dt\\
    & \leq \int_{-\infty}^{+\infty}e^{-\frac{t^2}{16}}dt - \frac{4}{\sqrt{n}}e^{-\frac{\sqrt{n}t}{4}}\Big|_{4\sqrt{n}}^{+\infty}\\
    & = 4\sqrt{\pi} + \frac{4}{\sqrt{n}}e^{-n^2}
\end{align*}
thus
\begin{align*}
    \mathbb{E}|Z| \lesssim 1/\sqrt{n}
\end{align*}

\textbf{3). Bound $\mathbb{E}\sup\limits_{\|x\|=1,\,\,x\in K}|x\T(\frac{\tilde{X}\T\tilde{X}}{n} - I)x|$}.\\
Let $\tilde{X_i}$ be the $i$-th row of the matrix $\tilde{X}$. Since $\tilde{X}$ is a standard Gaussian matrix, $\tilde{X_i}$ is a $p$-dimensional standard Gaussian vector. We have
\begin{align*}
    \mathbb{E}\sup\limits_{\|x\|=1,\,\,x\in K}|x\T(\frac{\tilde{X}\T\tilde{X}}{n} - I)x| & = \mathbb{E}\sup\limits_{\|x\|=1,\,\,x\in K}|x\T(\frac{1}{n}\sum_{i=1}^n\tilde{X_i}\tilde{X_i}\T - I)x|\\
    & = \mathbb{E}\sup\limits_{\|x\|=1,\,\,x\in K}|\frac{1}{n}\sum_{i=1}^n\langle \tilde{X_i},x\rangle^2 - \mathbb{E}\langle \tilde{X_i},x\rangle^2|
\end{align*}
Define a class of linear functionals $\mathcal{F} \coloneqq \{ \langle , v\rangle, \,\, v\in K\bigcap\bbs^{p-1}\}$.  The above quantity can be written as $\mathbb{E}\sup_{f\in\mathcal{F}}\,2\Big|\frac{1}{n}\sum_{i=1}^n f^2(\tilde{X_i})-\mathbb{E}f^2(\tilde{X_i}) \Big|$, which fits in the setting of Lemma \ref{emp_upper_bound}.

Now we decode the abstract RHS of Lemma \ref{emp_upper_bound} into an explicit expression related to the Gaussian Complexity $w(K\bigcap\bbs^{p-1})$. First we bound $d_{\psi_1}$. For a constant $c_1$
\begin{align*}
    d_{\psi_1} & = \sup_{v\in K\bigcap\bbs^{p-1}} \|\langle \tilde{X_i}, v\rangle\|_{\psi_1} \\
    & \leq c_1 \sup_{v\in K\bigcap\bbs^{p-1}} \|\langle \tilde{X_i}, v\rangle\|_{\psi_2}
\end{align*}
Since $\tilde{X_i}$ is a standard normal random variable, we have
\begin{align*}
    \|\langle \tilde{X_i}, v\rangle\|_{\psi_2} \leq \sqrt{\frac{8}{3}}\|v\|_2 \leq \sqrt{\frac{8}{3}}
\end{align*}
and so
\begin{align*}
    d_{\psi_1} & \leq c_1\sqrt{\frac{8}{3}}
\end{align*}

Then we bound $\gamma_2(\mathcal{F},\psi_2)$ by the Gaussian Complexity $w(K\bigcap\bbs^{p-1})$. Notice that the metric induced by the $\psi_2$-norm on $\mathcal{F}$ is equivalent with the Euclidean distance on $K\bigcap\bbs^{p-1}$
\begin{align*}
    \|(f-g)(\tilde{X_i})\|_{\psi_2} = \|\langle \tilde{X_i}, (v_f-v_g)\rangle\|_{\psi_2}
    \leq \sqrt{\frac{8}{3}} \|v_f-v_g\|_2
\end{align*}
where $f,g\in\mathcal{F}$ and $v_f,v_g\in K\bigcap\bbs^{p-1}$. By the first result of Lemma \ref{tala_lemmas}, for a constant $c_2$
\begin{align*}
    \gamma_2(\mathcal{F},\psi_2) \leq c_2\, \sqrt{\frac{8}{3}} \,\gamma_2(K\bigcap\bbs^{p-1}, \|\cdot\|_2)
\end{align*}
And by the second result of Lemma \ref{tala_lemmas}, for some constant $c_3$
\begin{align*}
    \gamma_2(\mathcal{F},\psi_2) \leq c_2 c_3 \sqrt{\frac{8}{3}}\,\,\mathbb{E} \sup_{v\in K\bigcap\bbs^{p-1}}\langle Z, v\rangle
\end{align*}
where $Z$ is a $p$-dimensional standard Gaussian random vector. And according to the definition of Gaussian Complexity $w(K\bigcap\mathbb{S}^{p-1})=\sup_{v\in K\bigcap\bbs^{p-1}}\langle Z, v\rangle$. Thus
\begin{align*}
    \mathbb{E}\sup\limits_{\|x\|=1,\,\,x\in K}|x\T(\frac{\tilde{X}\T\tilde{X}}{n} - I)x| & = \mathbb{E}\sup_{f\in\mathcal{F}}\,\Big|\frac{1}{n}\sum_{i=1}^n f^2(Z_i)-\mathbb{E}f^2(Z_i) \Big|\\
    & \leq c \big[d_{\psi_1}\frac{\gamma_2(\mathcal{F},\psi_2)}{\sqrt{n}} \vee\frac{\gamma_2^2(\mathcal{F},\psi_2)}{n}\big]\\
    & \leq cc_2c_3\frac{8}{3}\bigg[ \frac{c_1w(K\bigcap\mathbb{S}^{p-1})}{\sqrt{n}} \vee \frac{c_2c_3w^2(K\bigcap\mathbb{S}^{p-1})}{n}
    \bigg] \\
    & \lesssim \frac{w(K\bigcap\mathbb{S}^{p-1})}{\sqrt{n}} \vee \frac{w^2(K\bigcap\mathbb{S}^{p-1})}{n}
\end{align*}

\textbf{4). Bound $\mathbb{E}\omek$}.\\
Finally we can get
\begin{align*}
    \mathbb{E}\omek & \lesssim (5\nu_0+3)\bigg[ \frac{w(K\bigcap\mathbb{S}^{p-1})}{\sqrt{n}} \vee \frac{w^2(K\bigcap\mathbb{S}^{p-1})}{n} \bigg] + (\nu_0^2+\nu_0)\frac{1}{\sqrt{n}}
\end{align*}
Plug in $\nu_0 = \sqrt{\nu+1}-1$ to get
\begin{align*}
    \mathbb{E}\omek & \lesssim (5\sqrt{\nu+1}-2)\bigg[ \frac{w(K\bigcap\mathbb{S}^{p-1})}{\sqrt{n}} \vee \frac{w^2(K\bigcap\mathbb{S}^{p-1})}{n} \bigg] + \frac{\nu+3-3\sqrt{\nu+1}}{\sqrt{n}}\\
    & \lesssim \sqrt{\nu+1}\bigg[ \frac{w(K\bigcap\mathbb{S}^{p-1})}{\sqrt{n}} \vee \frac{w^2(K\bigcap\mathbb{S}^{p-1})}{n} \bigg] + \frac{\nu+3-3\sqrt{\nu+1}}{\sqrt{n}}
\end{align*}

\section{Proof of Theorem \ref{cvx_est} and Corollary \ref{cvx_est_coro}}
\subsection{Proof of Theorem \ref{cvx_est}}
Suppose first that $v\T  \bar{x} \geq 0$. 
\begin{align*}
    v\T Ev - \bx\T  E\bx & = 
    v\T Av\T  - v\T \bar{A}v - (\bx A\bx - \bx\bar{A}\bx) \\
    & \geq \bx\bar{A}\bx - v\T \bar{A}v \\
    & \geq \lambda - (v\T \bx)^2\lambda - [1-(v\T \bx)^2]\mu \qquad \qquad \quad \text{By 2. in Lemma \ref{ineqs}}\\
    & \geq \nu[1-(v\T \bx)^2] \\
    & \geq \frac{1}{2}\nu\|v-\bx\|^2
\end{align*}
For the first term, by the definition of $\omek$, we can directly get $v\T Ev - \bx\T  E\bx \leq 2\omek$. 
For the second term, notice that 
\begin{align*}
    v\T Ev - \bx\T  E\bx \leq 4\|v-\bx\|\omtan
    \qquad \qquad \quad 
    \text{By 3. in Lemma \ref{ineqs}}
\end{align*}
so 
\begin{align*}
    \frac{1}{2}\nu\|v-\bx\|^2 \leq4\|v-\bx\|\omtan \qquad \Rightarrow \qquad  \|v-\bar{x}\|\leq \frac{8\omtan}{\nu}
\end{align*}
Next suppose that $v\T  \bar{x} \leq 0$. Repeating the same proof as above we observe that
\begin{align*}
    v\T Ev - \bx\T  E\bx \geq \frac{1}{2}\nu\|v+\bx\|^2.
\end{align*}
This and the bound $v\T Ev - \bx\T  E\bx \leq 2\omek$ directly give the proof of one of the inequalities. Next by 3. in Lemma \ref{ineqs} we have 
\begin{align*}
    v\T Ev - \bx\T  E\bx \leq 4\|v+\bx\|\omek
\end{align*}
which implies the second bound.

\subsection{Proof of Corollary \ref{cvx_est_coro}}
It is sufficient to show $\omtan \vee \omek \leq \|E\|_{op}$. This is obvious by the definitions of $\omtan$ and $\omek$.

\section{Proof of Theorem \ref{stop} and Theorem \ref{bad_vt}}
A key to analyze the $L_2$ error of the cone projected power iteration estimator is: If $\bx\T v_0$ is larger than a certain constant, $\bx\T v_t$ is always larger than that constant throughout the iterations. This is rigorously presented in the following lemma \ref{bigger_c0}. It is a foundation for the proof of Theorem \ref{stop}. 

\begin{lemma}\label{bigger_c0}
Suppose $\bar{x}\T v_0\geq c_0$ for some $c_0 > c_{-1}$, and the first eigengap of $\bar{A}$ is greater than $(3 + 2\sqrt{2})\omek$, then $\bar{x}\T v_t\geq c_0\wedge c_{1}, \,\,\forall t\in\N$.
\end{lemma}
\begin{proof}[Proof of Lemma \ref{bigger_c0}]

We prove this result by induction. 

Given $\bx\T v_{t-1} \geq c_0$, if $\bar{x}\T v_t\geq \bar{x}\T v_{t-1}$, the inequality preserves trivially. 

Next suppose that $\bar{x}\T v_t\leq \bar{x}\T v_{t-1}$. We first show that $\bar{x}\T v_t$ is non-negative. We have the identity $\bar{x}\T v_t = \frac{\bar{x}\T \Pi_KAv_{t-1}}{\|\Pi_KAv_{t-1}\|}$, so it suffices to show that $\bar{x}\T \Pi_KAv_{t-1} \geq \bar{x}\T  Av_{t-1} \geq \lambda \bar{x}\T  v_{t-1} - \|E\|_K \geq 0$, which holds since $\bar{x}\T  v_{t-1} \geq c_0\wedge c_{1} > c_{-1} \geq \frac{\|E\|_K}{\lambda}$ (the last inequality is shown in Lemma \ref{lemma:bounds:c1:cm1}).  

Next observe the identities, 
\begin{align*}
    \bar{x}\T v_t & = \frac{\bar{x}\T \Pi_KAv_{t-1}}{\|\Pi_KAv_{t-1}\|} \,\geq \frac{\bar{x}\T Av_{t-1}}{v_t\T Av_{t-1}} \quad\quad\,\, \text{By Moreau's decomposition and 1. in Lemma \ref{ineqs}}\\
    & \geq \frac{\bar{x}\T Av_{t-1}}{v_t\T Av_t} \quad\quad\quad\quad\quad\quad
    \quad\quad\quad\text{ By 4. in Lemma \ref{ineqs}}\\
    & \geq \frac{\bar{x}\T \bar{A}v_{t-1}-\omek}{v_t\T \bar{A}v_t +\omek} \\
    & \geq \frac{\lambda \bar{x}\T v_{t-1} - \omek}
    {(\bar{x}\T v_t)^2\lambda + [1-(\bar{x}\T v_t)^2]\mu + \omek} 
    \qquad\qquad\qquad\qquad\quad \text{By 2. in Lemma \ref{ineqs}}\\
    & \geq \frac{\lambda \bar{x}\T v_{t-1} - \omek}
    {\nu(\bar{x}\T v_{t-1}) + \mu + \omek}
    \qquad\qquad\qquad\qquad\quad\qquad\qquad\quad \text{By the condition $0 \leq \bar{x}\T v_t\leq\bar{x}\T v_{t-1}$}
\end{align*}
Let $\alpha = \bar{x}\T v_{t-1}$, and $f(y) = \frac{\lambda y - \omek}{\nu y+\mu+\omek}$. Note that the roots of the quadratic equation $f(y) = y$ are $c_{-1}$ and $c_1$, so that for $y\in(-\infty, c_{-1}]\bigcup[c_1, \infty)$, we have $y \geq f(y)$.
By the fact that $\alpha = \bar{x}\T v_{t-1} \geq f(\alpha)$ and $\alpha \geq c_0\wedge c_{1} > c_{-1}$, it follows that $\alpha \geq c_{1}$. 
In addition since $\bar{x}\T  v_t \geq f(\alpha)$, it is sufficient to show $f(\alpha)\geq c_0\wedge c_{1}$. 
Notice that $f(y)$ is increasing, so $f(\alpha)\geq f(c_{1}) = c_{1} \geq c_0 \wedge c_{1}$. Thus the proof is complete.
\end{proof}

\subsection{Proof of Theorem \ref{stop}}
For the first term, we can start to derive a lower bound and an upper bound of $\bx\T Av_{t-1}$ based on some inequalities from Lemma \ref{ineqs},
\begin{align*}
    \bx\T Av_{t-1} & = \lambda\bx\T v_{t-1} + \bx\T Ev_{t-1} \geq
    \lambda\bx\T v_{t-1} - \omek\\
    \bx\T Av_{t-1} & = \bx\T \Pi_KAv_{t-1} + \bx\T \Pi_{K^\circ}Av_{t-1}
    \leq \bx\T \Pi_KAv_{t-1} = 
    \|\Pi_KAv_{t-1}\| \bx\T v_{t} = 
    (v_{t}\T Av_{t-1})\bx\T v_{t}
\end{align*}
Combine the above two inequalities to get
\begin{align}
\label{stop1}
    \lambda\bx\T v_{t-1} - \omek & \leq (v_{t}\T Av_{t-1})\bx\T v_{t} 
 = (v_{t}\T Av_{t-1})\bx\T v_{t-1} + 
    (v_{t}\T Av_{t-1})\bx\T (v_{t} - v_{t-1}) 
\end{align}
Since $\|v_{t} - v_{t-1}\|\leq \Delta$, by Cauchy-Schwartz inequality and the definition of $\omek$,
\begin{align*}
    (v_{t}\T Av_{t-1})\bx\T  (v_{t} - v_{t-1}) & \leq
    |v_{t}\T \bar{A}v_{t-1} + v_{t}\T Ev_{t-1}|\,\|v_{t} - v_{t-1}\|\\
    & \leq (\lambda + \omek)\Delta 
\end{align*}
Furthermore, with the use of results in Lemma \ref{ineqs}, 
\begin{align*}
    (v_{t}\T Av_{t-1})\bx\T v_{t-1} & = 
    (v_{t}\T \bar{A}v_{t-1} + v_{t}\T Ev_{t-1})\bx\T v_{t-1} \\
    & \leq 
    \big(
    v_{t}\T \bar{A}v_t + (v_{t-1}-v_t)\T \bar{A}v_t + \omek
    \big)\bx\T v_{t-1} \\
    & \leq (v_{t}\T \bar{A}v_t)\bx\T v_{t-1} +
    \big(\|v_{t-1}-v_t\|\,\|\bar{A}v_t\| + \omek\big) \bx\T v_{t-1} \\
    & \leq (v_{t}\T \bar{A}v_t)\bx\T v_{t-1} +
    \big(\Delta\lambda + \omek\big) \bx\T v_{t-1} \\
    & \leq \big[
    \lambda(\bx\T v_t)^2 + \mu\big(1-(\bx\T v_t)^2\big)
    \big]\bx\T v_{t-1} + 
    \big(\Delta\lambda + \omek\big) \bx\T v_{t-1}
\end{align*}
Thus (\ref{stop1}) becomes
\begin{align*}
    \lambda\bx\T v_{t-1} - \omek & \leq 
    \big[
    \lambda(\bx\T v_t)^2 + \mu\big(1-(\bx\T v_t)^2\big)
    \big]\bx\T v_{t-1} + 
    \big(\Delta\lambda + \omek\big) \bx\T v_{t-1} +
    (\lambda + \omek)\Delta \\
    \Rightarrow \quad 
    \nu[1-(\bx\T v_t)^2]\bx\T v_{t-1} & \leq
    \big(\Delta\lambda + \omek\big) \bx\T v_{t-1} +
    (\lambda + \omek)\Delta + \omek 
\end{align*}
According to Lemma \ref{bigger_c0}, $\bx\T v_t\geq c_0\wedge c_{1} \geq 0$, so $1-(\bx\T v_t)^2\geq 1-\bx\T v_t = \frac{1}{2}\|v_t-\bx\|^2$, thus
\begin{align*}
    \frac{\nu}{2}\|v_t-\bx\|^2\bx\T v_{t-1} & \leq 
    \big(\Delta\lambda + \omek\big) \bx\T v_{t-1} +
    (\lambda + \omek)\Delta + \omek
\end{align*}
Since $c_0\wedge c_{1}\leq\bx\T v_{t-1}\leq 1$ and $\Delta \leq \frac{\omek}{2\lambda}\wedge1$, 
\begin{align*}
    \frac{\nu}{2}\|v_t-\bx\|^2(c_0\wedge c_{1}) & \leq 
    \big(\Delta\lambda + \omek\big) +
    (\lambda + \omek)\Delta + \omek \\
    & \leq 4\omek \\
    \Rightarrow \quad 
    \|v_t-\bx\| & \leq 
    \sqrt{\frac{8\omek}{(c_0\wedge c_{1})\nu}}
\end{align*}

For the second part, first to get a lower bound of $\bx\T Av_{t-1}$,
\begin{align}
\label{thm2_tan1}
    \bx\T Av_{t-1} & = \bx\T \bar{A}v_{t-1} + \bx\T Ev_{t-1}\nonumber\\
    & = \lambda\bx\T v_{t-1} + \bx\T Ev_{t-1}\nonumber\\
    & = \lambda\bx\T v_{t-1} + \bx\T E\bx + (v_{t-1}-\bx)\T E(v_{t-1}-\bx) + 2\bx\T E(v_{t-1}-\bx)\nonumber\\
    & \geq \lambda\bx\T v_{t-1} + \bx\T E\bx - \|v_{t-1}-\bx\|^2\omtan - 2\|v_{t-1}-\bx\|\omtan \nonumber\\
    & \geq \lambda\bx\T v_{t-1} + \bx\T E\bx - 4\|v_{t-1}-\bx\|\omtan \quad\quad \text{By the fact }\|v_{t-1}-\bx\|\leq2\nonumber\\
    & \geq \lambda\bx\T v_{t-1} + \bx\T E\bx - 4(\|v_{t}-\bx\| + \Delta)\omtan \nonumber\\
    & = \lambda\bx\T v_{t-1} + \bx\T E\bx - 4\|v_{t}-\bx\|\omtan-4\Delta\omtan
\end{align}
And again, use the second and third results in Lemma \ref{ineqs} to get an upper bound of $\bx\T Av_{t-1}$,
\begin{align}
\label{thm2_tan2}
    \bx\T Av_{t-1} & \leq \bx\T \Pi_KAv_{t-1} = 
    (v_{t}\T Av_{t-1})\bx\T v_{t}\nonumber\\
    & \leq(v_{t}\T Av_{t-1})\bx\T v_{t-1} + (v_{t}\T Av_{t-1})\|v_{t}-v_{t-1}\|\nonumber\\
    & \leq [v_t\T Av_t+ 
    (v_{t-1}-v_t)\T Av_t]\bx\T v_{t-1} + 
    (v_{t}\T Av_{t-1})\Delta\nonumber\\
    & = [v_t\T \bar{A}v_t + v_t\T Ev_t + 
    (v_{t-1}-v_t)\T Av_t]\bx\T v_{t-1} + 
    (v_{t}\T Av_{t-1})\Delta\nonumber\\
    & \leq \Big[\lambda(\bx\T v_t)^2+\mu(1-(\bx\T v_t)^2)\Big]\bx\T v_{t-1} +
    \Big[v_t\T Ev_t\Big]\bx\T v_{t-1} +
    \Big[(v_{t-1}-v_t)\T Av_t\Big]\bx\T v_{t-1} +
    (v_{t}\T Av_{t-1})\Delta
\end{align}
Then we bound each term in the RHS of the above inequality. First, by the third result in Lemma \ref{ineqs}
\begin{align}
\label{thm2_tan3}
    \Big[v_t\T Ev_t\Big]\bx\T v_{t-1} &= 
    \Big[ \bx\T E\bx + (v_t-\bx)\T E(v_t-\bx) + 2\bx\T E(v_t-\bx) \Big]\bx\T v_{t-1}\nonumber\\
    &\leq (\bx\T E\bx)\bx\T v_{t-1} + |(v_t-\bx)\T E(v_t-\bx) + 2\bx\T E(v_t-\bx)| |\bx\T v_{t-1}| \qquad \text{Since } \bx\T v_{t-1} \geq 0 \nonumber\\
    &\leq (\bx\T E\bx)\bx\T v_{t-1} + (\|v_t-\bx\|^2 + 2\|v_t-\bx\|)\omtan|\bx\T v_{t-1}| \nonumber\\
    &\leq (\bx\T E\bx)\bx\T v_{t-1} + 4\|v_t-\bx\|\omtan|\bx\T v_{t-1}| \nonumber\\
    &\leq (\bx\T E\bx)\bx\T v_{t-1} + 4\|v_t-\bx\|\omtan
\end{align}
With some modifications and the Cauchy-Schwartz inequality, the second term can be bounded as
\begin{align*}
    \Big[(v_{t-1}-v_t)\T Av_t\Big]\bx\T v_{t-1} &= \Big[(v_{t-1}-v_t)\T \bar{A}v_t + (v_{t-1}-v_t)\T Ev_t\Big]\bx\T v_{t-1}\\
    &\leq \Big[\Delta\lambda + |(v_{t-1}-v_t)\T Ev_t|\Big]|\bx\T v_{t-1}|\\
    &\leq \Delta\lambda + |(v_{t-1}-\bx)\T Ev_t| + |(v_{t}-\bx)\T Ev_t|\\
    &\leq \Delta\lambda + |(v_{t-1}-\bx)\T E(v_t-\bx) + (v_{t-1}-\bx)\T E\bx|\\
    &+ |(v_{t}-\bx)\T E(v_t-\bx) + (v_{t}-\bx)\T E\bx|\\
    &\leq \Delta\lambda + (\|v_{t-1}-\bx\|\|v_{t}-\bx\| \\
    & + \|v_{t-1}-\bx\|)\omtan + (\|v_{t}-\bx\|^2 + \|v_{t}-\bx\|)\omtan
\end{align*}
Use the fact that $\|v_{t}-\bx\|\leq2$ for any $t$, the above inequality can be written as
\begin{align}
\label{thm2_tan4}
    \Big[(v_{t-1}-v_t)\T Av_t\Big]\bx\T v_{t-1} &\leq \Delta\lambda + 3\|v_{t-1}-\bx\|\omtan + 3\|v_{t}-\bx\|\omtan\nonumber\\
    &\leq \Delta\lambda + 3(\|v_{t}-\bx\| + \Delta)\omtan + 3\|v_{t}-\bx\|\omtan\nonumber\\
    &= \Delta\lambda + 3\Delta\omtan + 6\|v_{t}-\bx\|\omtan
\end{align}
For the last term, observe that
\begin{align*}
    (v_{t}\T Av_{t-1})\Delta &= (v_{t}\T \bar{A}v_{t-1} + v_{t}\T Ev_{t-1})\Delta\\
    &\leq \Big[ \lambda + \bx\T E\bx + (v_{t-1}-\bx)\T E(v_t-\bx) + \bx\T E(v_{t-1}-\bx) + \bx\T E(v_t-\bx)\Big]\Delta\\
    &= \Big[ \lambda + \bx\T E\bx + (\|v_{t-1}-\bx\|\|v_t-\bx\| + \|v_{t-1}-\bx\| + \|v_t-\bx\|)\omtan \Big]\Delta
\end{align*}
Use the fact that $\|v_{t}-\bx\|\leq2$ for any $t$, and $\|v_t - v_{t-1}\|\leq \Delta$, the above inequality can be written as
\begin{align}
\label{thm2_tan5}
    (v_{t}\T Av_{t-1})\Delta &\leq \Big[ \lambda + \bx\T E\bx + (2\|v_t-\bx\| + \Delta+\|v_{t}-\bx\| + \|v_t-\bx\|)\omtan \Big]\Delta\nonumber\\
    &= \Big[ \lambda + \bx\T E\bx + (4\|v_t-\bx\| + \Delta)\omtan \Big]\Delta\nonumber\\
    &\leq \Delta\lambda + \Delta\omtan + 4\omtan\|v_t-\bx\| + \Delta^2\omtan\nonumber\\
    &\leq \Delta\lambda + 2\Delta\omtan + 4\omtan\|v_t-\bx\|
\end{align}
Combine the inequalities \eqref{thm2_tan1}, \eqref{thm2_tan2} to get \begin{align*}
    (\lambda-\mu)[1-(\bx\T v_t)^2]\bx\T v_{t-1} + \bx\T E\bx - 4\|v_{t}-\bx\|\omtan-4\Delta\omtan \leq\\ \Big[v_t\T Ev_t\Big]\bx\T v_{t-1} +
    \Big[(v_{t-1}-v_t)\T Av_t\Big]\bx\T v_{t-1} +
    (v_{t}\T Av_{t-1})\Delta
\end{align*}
Plug in \eqref{thm2_tan3}, \eqref{thm2_tan4}, \eqref{thm2_tan5} into the RHS of above inequality to get
\begin{align*}
    (\lambda-\mu)[1-(\bx\T v_t)^2]\bx\T v_{t-1} + (1-\bx\T v_{t-1})\bx\T E\bx \leq 18\omtan\|v_t-\bx\| + 9\Delta\omtan + 2\Delta\lambda
\end{align*}
By the fact that 
$1-\bx\T v_{t-1} = \frac{1}{2}\|v_{t-1}-\bx\|^2 \leq \frac{1}{2}(\Delta + \|v_t-\bx\|)^2 = \frac{1}{2}(\Delta^2 + 2\Delta\|v_t-\bx\| + \|v_t-\bx\|^2) \leq \Delta+2\|v_t-\bx\|$, the above inequality becomes
\begin{align*}
    &(\lambda-\mu)[1-(\bx\T v_t)^2]\bx\T v_{t-1} - (\Delta + 2\|v_t-\bx\|)\omtan \leq 18\omtan\|v_t-\bx\| + 9\Delta\omtan + 2\Delta\lambda\\
    \Rightarrow \quad &(\lambda-\mu)[1-(\bx\T v_t)^2]\bx\T v_{t-1} - 20\omtan\|v_t-\bx\| - 10\Delta\omtan - 2\Delta\lambda
    \leq 0
\end{align*}
Use the fact $c_0\wedge c_{1}\leq \bx\T v_{t-1}\leq 1$, and $\lambda-\mu=\nu$, the above inequality becomes
\begin{align}
\label{thm2_tan6}
    \frac{\nu(c_0\wedge c_{1})}{2}\|v_t-\bx\|^2-20\omtan\|v_t-\bx\| - 10\Delta\omtan - 2\Delta\lambda
    \leq 0 
\end{align}
The above inequality is a quadratic form of $\|v_t-\bx\|$. The discriminant $D = (20\omtan)^2 + 4\nu\Delta(c_0\wedge c_{1})(5\omtan+\lambda)\geq 0$, so there are values of $\|v_t-\bx\|$ to make \eqref{thm2_tan6} hold. 
By calculating the roots we get
\begin{align*}
    \|v_t-\bx\| \leq \frac{20\omtan + \sqrt{D}}{\nu(c_0\wedge c_{1})}
\end{align*}
Then we need to pick a suitable $\Delta$ such that $\sqrt{D}$ is the same order as $\omtan$. When $\Delta \leq \frac{\omtan}{(c_0\wedge c_{1})\nu} \wedge \frac{4\omtan^2}{(c_0\wedge c_{1})\lambda\nu}$, we have 
\begin{align*}
    20\nu\Delta(c_0\wedge c_{1})\omtan &\leq 20\omtan^2\\
    4\lambda\nu\Delta(c_0\wedge c_{1}) &\leq 16\omtan^2
\end{align*}
Thus
\begin{align*}
    D \leq 400\omtan^2 + 36\omtan^2 \leq (21\omtan)^2
\end{align*}
and 
\begin{align*}
    \|v_t-\bx\| \leq \frac{20\omtan + \sqrt{D}}{\nu(c_0\wedge c_{1})}
    \leq 
    \frac{41\omtan}{\nu(c_0\wedge c_{1})}
\end{align*}
Since the initial vector $v_0$ satisfies $v_0\T\bx\geq c_0$, it is always true that $v_t\T\bx\geq c_0$. Thus $\|v_t-\bx\|$ is always smaller than $\|v_t+\bx\|$. We finally get
\begin{align*}
    \|v_t-\bx\|\wedge\|v_t+\bx\| \leq \sqrt{\frac{8\omek}{(c_0\wedge c_{1})\nu}}\wedge\frac{41\omtan}{(c_0\wedge c_{1})\nu}
\end{align*}

\subsection{Proof of Theorem \ref{bad_vt}}
\begin{align*}
    v_t\T Av_t & = v_t\T \bar{A}v_t + v_t\T Ev_t
    \geq \lambda(v_t\T \bx)^2 - \omek\\
    \tvt\T A\tvt & = \tvt\T \bar{A}\tvt + \tvt\T E\tvt \nonumber \\
    & \leq \lambda(\tvt\T \bx)^2 + \mu[1-(\tvt\T \bx)^2] + \omek
    \qquad \qquad \text{By 2. in Lemma \ref{ineqs}}
\end{align*}
By the above inequalities, when $\tvt\T A\tvt \geq v_t\T Av_t$, we have
\begin{align*}
    \lambda(v_t\T \bx)^2 - \omek
    & \leq 
    \lambda(\tvt\T \bx)^2 + \mu[1-(\tvt\T \bx)^2] + \omek \nonumber\\
    \Rightarrow \quad 
    \lambda(v_t\T \bx)^2-\lambda - 2\omek
    & \leq 
    \lambda(\tvt\T \bx)^2 -\lambda + \mu[1-(\tvt\T \bx)^2] \nonumber\\
    \Rightarrow \quad 
    (\lambda-\mu)[1-(\tvt\T \bx)^2] & \leq
    \lambda[1-(v_t\T \bx)^2] + 2\omek
\end{align*}
Notice that
\begin{align*}
    1-(v_t\T \bx)^2 = (1+v_t\T \bx)(1-v_t\T \bx) \leq 2(1-v_t\T \bx)
    = \|v_t - \bx\|^2
\end{align*}
Thus the previous inequality becomes
\begin{align}
    \nu[1-(\tvt\T \bx)^2] &\leq
    \lambda\|v_t - \bx\|^2 + 2\omek\nonumber\\
    \Rightarrow \quad
    \frac{\nu}{2}\|\tvt-\bx\|^2 &\leq
    \lambda\|v_t - \bx\|^2 + 2\omek\nonumber\\
    \Rightarrow \quad \|\tvt-\bx\|\wedge \|\tvt+\bx\| &\leq 
    \sqrt{
    \frac{2\lambda\|v_t-\bx\|^2}{\nu} + 
    \frac{4\omek}{\nu}
    }\nonumber\\
    & \leq 
    \sqrt{\frac{2\lambda}{\nu}}\|v_t-\bx\|+
    \sqrt{\frac{4\omek}{\nu}}
    \label{bvt_final}
\end{align}

By the result 2 and 3 in Lemma \ref{ineqs}, we get 
\begin{align}
    v_t\T Av_t -\bx\T E\bx & = v_t\T \bar{A}v_t + v_t\T Ev_t -\bx\T E\bx
    \geq \lambda(v_t\T \bx)^2 - 4\|v_t-\bx\|\omtan\label{vt_lower2}\\
    \tvt\T A\tvt -\bx\T E\bx & = \tvt\T \bar{A}\tvt + \tvt\T E\tvt -\bx\T E\bx \nonumber \\
    & \leq \lambda(\tvt\T \bx)^2 + \mu[1-(\tvt\T \bx)^2] + 4\|\tvt-\bx\|\omtan\label{bvt_upper2}\\
    \text{Or }\qquad \qquad 
    & \leq \lambda(\tvt\T \bx)^2 + \mu[1-(\tvt\T \bx)^2] + 4\|\tvt+\bx\|\omek\label{bvt_upper2_1} 
\end{align}
If $\tvt\T A\tvt \geq v_t\T Av_t$, and $\tvt\T \bx \geq 0$ such that $1-(\tvt\T \bx)^2 \geq \frac{1}{2}\|\tvt-\bx\|^2$, by (\ref{vt_lower2}), (\ref{bvt_upper2}) we have
\begin{align*}
    (\lambda-\mu)[1-(\tvt\T \bx)^2] - \lambda[1-(v_t\T \bx)^2] - 4\|v_t-\bx\|\omtan - 4\|\tvt-\bx\|\omtan &\leq 0\\
    \Rightarrow \quad
    \frac{\nu}{2}\|\tvt-\bx\|^2 -
    4\|\tvt-\bx\|\omtan  - 
    \lambda[1-(v_t\T \bx)^2] - 
    4\|v_t-\bx\|\omtan  
    &\leq 0
\end{align*}
Notice that
\begin{align*}
    1-(v_t\T \bx)^2 = (1+v_t\T \bx)(1-v_t\T \bx) \leq 2(1-v_t\T \bx)
    = \|v_t - \bx\|^2
\end{align*}
Thus the previous inequality becomes
\begin{align*}
    \frac{\nu}{2}\|\tvt-\bx\|^2 -
    4\omtan\|\tvt-\bx\|  - 
    \lambda\|v_t\T-\bx\|^2 - 
    4\|v_t-\bx\|\omtan  
    &\leq 0
\end{align*}
The discriminant $D = 16\omtan^2 + 
2\nu\big[\lambda\|v_t - \bx\|^2 + 4\|v_t-\bx\|\omtan\big] > 0$. Thus 
\begin{align*}
    \|\tvt-\bx\| & \leq \frac{4\omtan + \sqrt{D}}{\nu}\\
    & \leq \frac{8\omtan}{\nu} + 
    \sqrt{\frac{2\lambda\|v_t - \bx\|^2 + 8\|v_t-\bx\|\omtan}{\nu}}
\end{align*}
Similarly, if $\tvt\T A\tvt \geq v_t\T Av_t$, and $\tvt\T \bx \leq 0$ such that $1-(\tvt\T \bx)^2 \geq \frac{1}{2}\|\tvt+\bx\|^2$, by (\ref{vt_lower2}), (\ref{bvt_upper2_1}) we have 
\begin{align*}
    \frac{\nu}{2}\|\tvt+\bx\|^2 -
    4\omek\|\tvt+\bx\|  - 
    \lambda\|v_t\T-\bx\|^2 - 
    4\|v_t-\bx\|\omtan  
    &\leq 0
\end{align*}
The discriminant $D = 64\omek^2 + 
2\nu[\lambda\|v_t-\bx\|^2 + 4\|v_t-\bx\|\omtan] > 0$. Thus 
\begin{align*}
    \|\tvt+\bx\| & \leq \frac{4\omek + \sqrt{D}}
    {\nu}\\
    & \leq \frac{8\omek}{\nu} + 
    \sqrt{\frac{2\lambda\|v_t-\bx\|^2 + 8\|v_t-\bx\|\omtan}{\nu}}
\end{align*}
Thus 
\begin{align}
\label{bvt_final2}
    \|\tvt-\bx\| \wedge \|\tvt+\bx\| \leq
    \Big( 
    \frac{8\omek}{\nu} \vee 
    \frac{8\omtan}{\nu}
    \Big) + 
    \sqrt{\frac{2\lambda\|v_t-\bx\|^2 + 8\|v_t-\bx\|\omtan}{\nu}}
\end{align}
Therefore $\|\tvt-\bx\| \wedge \|\tvt+\bx\|$ is bounded by the minimum of RHS of \eqref{bvt_final} and \eqref{bvt_final2}.

\section{Proof of Theorem \ref{minimax} and Corollary \ref{minimax_coro}}

In order to prove Theorem \ref{minimax}, we need to introduce three intermediate results. 
The first one is about trace of the product of inverse spike matrices, which will be used in calculating the Kullback-Leibler divergence of multivariate Gaussian. 
\begin{lemma} If $\|v\| = 1, \|v'\|=1$, for a constant $k$
\label{trace_eq}
\begin{align*}
    \text{tr}\big( (I+kv'{v'}\T )^{-1} (I+kvv\T )\big) = p +
  \frac{k^2}{k+1}[1-\langle v, v'\rangle^2 ]
\end{align*}
\end{lemma}
\begin{proof}[Proof of Lemma \ref{trace_eq}]
By the Sherman-Morrison formula, 
\begin{align*}
    (I+kv'{v'}\T )^{-1} = 
    I - \frac{k}{k+1}v'{v'}\T 
\end{align*}
Thus 
\begin{align*}
    \text{tr}\big( (I+kv'{v'}\T )^{-1} (I+kvv\T )\big) & = 
    \text{tr}\big( (I - \frac{k}{k+1}v'{v'}\T )
    (I+kvv\T )\big) \\
    & = \text{tr}\big( 
    I + kvv\T  - \frac{k}{k+1}v'{v'}\T  - 
    \frac{k^2}{k+1}v'{v'}\T vv\T 
    \big) \\
    & = p + k - \frac{k}{k+1} - \frac{k^2}{k+1}\langle v, v'\rangle^2\\
    & = p + \frac{k^2}{k+1}[1-\langle v, v'\rangle^2 ]
\end{align*}
\end{proof}

The second intermediate result Lemma \ref{sudakov_vari} upper bounds the Gaussian complexity of a set in terms of its metric entropy. Gaussian complexity is actually a basic geometric property. Lemma \ref{gw_property} reveals a relation between Gaussian complexity and diameter of a set, which is used in the proof of Lemma \ref{sudakov_vari}. 
\begin{lemma}
\label{gw_property}
\cite[Proposition~7.5.2]{vershynin2018high}
(Gaussian Complexity and Diameter)
Let $T\subset \mathbb{R}^p$, and $w(T)$ is the Gaussian Complexity of $T$. Then 
\begin{align*}
    \frac{1}{\sqrt{2\pi}} \leq \frac{w(T)}{\operatorname{diam}(T)} \leq \frac{\sqrt{p}}{2}
\end{align*}
\end{lemma}

Recall that the \textit{$\epsilon$-covering number} of a set $T$ is the cardinality of the smallest $\epsilon$-cover of $T$, and the logarithm of it is the metric entropy. The \textit{$\epsilon$-packing number} is the cardinality of the largest $\epsilon$-packing of $T$. The $\epsilon$-covering and $\epsilon$-packing of a set are in the same order \cite[Lemma~5.1]{wainwright2019high}. The following lemma gives upper bound of the Gaussian complexity using a chaining constructed by the covering sets. 

\begin{lemma}
\label{sudakov_vari}
\textbf{(Variation of Reverse Sudakov's Inequality)}
Let $T$ be a subset of $\mathbb{R}^p$ with finite diameter. 
Let $\log N_{\epsilon}$ be the metric entropy of $T$ with respect to $\epsilon$, and $w(T)$ is the Gaussian complexity of $T$ and $w(T) > 64\sqrt{\log 3}$. There exists a constant $C$, such that
\begin{align*}
    w(T) \leq 64\sqrt{\log 3} + C\log p \sup\limits_{\epsilon\geq 0, \,\,N_{\epsilon}\geq  4}\epsilon\sqrt{\log N_{\epsilon}}
\end{align*}
\end{lemma}
\begin{proof}[Proof of Lemma \ref{sudakov_vari}]
Let $X_t = \langle g,t\rangle$ where $t\in T,\,\,g\sim \mathcal{N}(0,I_p)$ to be a Gaussian process on $T$. By definition, $w(T) = \sup_{t\in K}\, X_t$.
\begin{enumerate}
    \item \textbf{Construct a chaining $\{\pi_{\nu}(t), ..., \pi_n(t)\}$ on $T$.} \\
    Let $\epsilon_i = 2^{-i},\, i\in\mathbb{Z}$, and $\mathcal{T}(T,\epsilon_i)$ is the min $\epsilon_i$-covering of $T$, and $|\mathcal{T}(T,\epsilon_i)| = N_{\epsilon_i}$. For any $t\in T$, let $\pi_i(t)$ be a point in $\mathcal{T}(T,\epsilon_i)$ satisfying $\|t-\pi_i(t)\|_2 \leq \epsilon_i$.
    
    We start with a covering which has only one point. Let $\nu = \max\limits_{i\in\mathbb{Z}}\{i:\,\,\epsilon_{i} \geq \operatorname{diam}(T)\}$, so $\mathcal{T}(T,\epsilon_{\nu}) = \{\pi_{\nu}(t)\}$. Since $\operatorname{diam}(T) \leq 2$ we have that $\nu \geq -1$. Next we choose $n$ such that $X_{\pi_n(t)}$ is close enough to $X_t$:
    \begin{align}
    \label{n_bound}
        n = \min_{i\in\mathbb{Z}} \{i:\,\,\frac{1}{2^i}\leq \frac{w(T)}{2\sqrt{p}}\}
    \end{align}
    Then 
    \begin{align*}
        X_t - X_{\pi_{\nu}(t)} =
        \sum\limits_{i=\nu}^n (X_{\pi_i(t)} - X_{\pi_{i-1}(t)}) + 
        (X_t - X_{\pi_n(t)})
    \end{align*}
    Since $\pi_{\nu}(t)$ doesn't depend on $t$, we have $\mathbb{E}X_{\pi_{\nu}(t)}=0$, thus 
    \begin{align}
    \label{sudakov_main}
        w(T) = \E\sup_{t\in T}X_t \leq 
        \E\sum\limits_{i=\nu}^n \sup_{t\in T}(X_{\pi_i(t)} - X_{\pi_{i-1}(t)}) + 
        \E\sup_{t\in T}(X_t - X_{\pi_n(t)})
    \end{align}
    
    \item \textbf{Upper bound (\ref{sudakov_main}) by the properties of chaining}. \\
    For a given $t\in T$, it is easy to see $X_{\pi_i(t)} - X_{\pi_{i-1}(t)}$ is a Lipschitz function of standard normal vector with Lipschitz factor $L=\|\pi_i(t)-\pi_{i-1}(t)\|_2 \leq 2\epsilon_{i-1}$. Thus 
    \begin{align*}
        X_{\pi_i(t)} - X_{\pi_{i-1}(t)} \sim
        SG(4\epsilon_{i-1}^2)
    \end{align*}
    The expected maximum of $N$ $SG(\nu^2)$ random variables is at most $\nu\sqrt{2\log N}$ \citep[page 31]{boucheron2013concentration}. We have $N=N_{\epsilon_i}N_{\epsilon_{i-1}}$, so that
    \begin{align}
    \label{sudakov_1}
        \E\sup_{t\in T}(X_{\pi_i(t)} - X_{\pi_{i-1}(t)}) & \leq  2\epsilon_{i-1} \sqrt{2\log(N_{\epsilon_i}N_{\epsilon_{i-1}})} \nonumber \\
        & \leq 4\epsilon_{i-1}
        \sqrt{\log N_{\epsilon_i}}
    \end{align}
    
    For the second term, 
    \begin{align}
    \label{sudakov_2}
        \E\sup_{t\in T}(X_t - X_{\pi_n(t)}) & = 
        \E\sup_{t\in T}\,\, \langle g, t - \pi_n(t)\rangle \nonumber\\
        & \leq \epsilon_n\E\|g\|_2 \nonumber\\
        & \leq \frac{w(T)}{2\sqrt{p}}\sqrt{p}
        \,\,= \frac{1}{2}w(T)
    \end{align}
    
      Combining (\ref{sudakov_main}), (\ref{sudakov_1}), (\ref{sudakov_2}) get 
    \begin{align}
    \label{sudakov_main_final}
        \frac{1}{2}w(T) & \leq 4
        \sum\limits_{i=\nu}^n \epsilon_{i-1}\sqrt{\log N_{\epsilon_i}} 
    \end{align}
    
    Notice that $n-\nu = \mathcal{O}(\log p)$ steps should be sufficient to walk from $X_{\pi_{\nu}(t)}$ to $X_t$: 
    \begin{align*}
        n - \nu - 1& = \log_2  \frac{\epsilon_{\nu+1}}{\epsilon_n}\nonumber\\
        & \leq \log_2 \frac{\operatorname{diam}(T)}{w(T)/4\sqrt{p}} 
        \qquad\qquad \text{Since }\epsilon_{\nu+1} < \operatorname{diam}(T) \text{ and } \epsilon_n \geq \frac{w(T)}{4\sqrt{p}}
        \nonumber\\
        & \leq \log_2 (4\sqrt{2\pi}\sqrt{p})
        \qquad \qquad \text{Lemma \ref{gw_property}} 
        \nonumber\\
        & = C \log p
    \end{align*}
    Let $i^* = \min\limits_{i\in[\nu,n]}\{ N_{\epsilon_i} \geq 4\}$. We can always find such an $i^*$, since if all $N_{\epsilon_i} \leq 3$ by \eqref{sudakov_main_final} we have
    \begin{align*}
        \frac{1}{2}w(T) & \leq 4
        \sum\limits_{i=\nu}^n \epsilon_{i-1}\sqrt{\log 3} \leq 32 \sqrt{\log 3} ,
    \end{align*}
    where in the last inequality we used the fact that $\nu \geq -1$. This is a contradiction with our assumption. 
    With such an $i^*$ the inequality (\ref{sudakov_main_final}) can be written as 
    \begin{align*}
        \frac{1}{8}w(T) & \leq 
        \sum\limits_{i=\nu}^{i^*-1} 2^{-(i-1)}\sqrt{\log 3} +
        (n + 1 - i^*) \sup\limits_{\epsilon\geq 0, \,\,N_{\epsilon}\geq 4}\epsilon\sqrt{\log N_{\epsilon}} \\
        & \leq 8\sqrt{\log 3} + 
        C\log p\sup\limits_{\epsilon\geq 0, \,\,N_{\epsilon}\geq 4}\epsilon\sqrt{\log N_{\epsilon}}
    \end{align*}
\end{enumerate}
\end{proof}

The last intermediate result we need is the generalized Fano's inequality for multi sample setting. Suppose that we know a random variable $Y$ and want to estimate another random variable $X$ based on $Y$. Fano's inequality quantifies the estimation uncertainty in terms of the conditional entropy of $X$ on $Y$ \cite[Page~37]{cover2012elements}. The generalized Fano's method, derived from the original Fano's inequality, is widely used in statistics literature to provide a lower bound of the estimation error \cite[Lemma~3]{yu1997assouad}. \citet[Lemma~3]{yu1997assouad} proved the generalized Fano's inequality for a single sample setting. With the same idea, we can get a multi sample version of generalized Fano's method 
\begin{lemma}(Generalized Fano's Method) 
\label{fano_yu}
Let $r\geq 2$ be an integer and let $\mathcal{M}_r\subset\mathcal{P}$ contain $r$ probability measures indexed by $j=1,2,...,r$. Let $X_1,...,X_n$ be a collection of i.i.d. random variables with the conditional distribution $P_j$ when $j$ is given. $\theta(P_j)$ is a parameter of $P_j$, and $\hat{\theta}$ is an estimation based on the sample $X_1,...,X_n$.
The probability measures in $\mathcal{M}_r$ also satisfy that for all $j\neq j'$
\begin{align*}
    d(\theta(P_j), \theta(P_{j'})) \geq \alpha_r,
\end{align*}
and 
\begin{align*}
    D(P_j||P_{j'}) \leq \beta_r
\end{align*}
Then 
\begin{align*}
    \max\limits_{j}\mathbb{E}\big[d(\hat{\theta}, \theta(P_j))\big]
    \geq 
    \frac{\alpha_r}{2}(1-\frac{n\beta_r + \log2}{\log r})
\end{align*}
\end{lemma}

Based on the previous three intermediate results, we are able to analyze the lower bound of the $L_2$ error rate of the eigenvector estimator in spiked covariance model. The lower bound is related to the Gaussian complexity of $K\bigcap \mathbb{S}^{p-1}$, where $K$ is the convex cone containing the true eigenvector $\bx$.

\subsection{Proof of Theorem \ref{minimax}}
\begin{enumerate}
    \item Let $\mathcal{P}(K\bigcap\bbs^{p-1}, \epsilon)$ be the max  $\epsilon$-packing of $K\bigcap\bbs^{p-1}$ such that for any $v_i, v_j\in \mathcal{P}(K\bigcap\bbs^{p-1},\epsilon)
    \,\text{ and }
    v_i\neq v_j$ we have $\|v_i-v_j\|\geq \epsilon$. Let the $\epsilon$-packing number be $M_{\epsilon} = |\mathcal{P}(K\bigcap\bbs^{p-1}, \epsilon)|$.

    \item Let $P_v = \mathcal{N}(0, I + \nu vv\T)$. By Lemma \ref{trace_eq}, the Kullback-Leibler distance between $P_{v_i}$ and $P_{v_j}$ for any $v_i, v_j\in \mathcal{P}(K\bigcap\bbs^{p-1})
    \,\text{ and }
    v_i\neq v_j$ is  
    \begin{align*}
        D(P_{v_i}\|P_{v_j}) & = \frac{1}{2}\Big[
        \text{tr}\big( (I+\nu v_jv_j\T )^{-1} (I+\nu v_iv_i\T )\big) - p\Big] \\
        & = \frac{\nu^2}{2(\nu+1)}[1-\langle v_i, v_j\rangle^2]
        \coloneqq D
    \end{align*}
    
    \item By Lemma \ref{fano_yu}, and $D = \frac{\nu^2}{2(\nu+1)}[1-\langle v_i, v_j\rangle^2] \leq \frac{\nu^2}{2(\nu+1)} \lesssim \frac{\nu\wedge \nu^2}{2}$, we have
    \begin{align*}
    \inf\limits_{\hat{v}}\max\limits_{v_j\in \mathcal{P}(K\bigcap\bbs^{p-1}, \epsilon)} 
    \E\|\hat{v} - v_j\| &\geq 
    \frac{\epsilon}{2}
    (1-\frac{ nD+\log2}{\log M_{\epsilon}}) \\
    & \gtrsim \frac{\epsilon}{2}
    (1-\frac{n (\nu\wedge \nu^2)}{2\log M_{\epsilon}}
    -\frac{\log 2}{\log M_{\epsilon}}) 
    \end{align*}
Let $N_{\epsilon}$ be the $\epsilon$-covering number of set $K\bigcap\bbs^{p-1}$. By the fact $N_{\epsilon}\leq M_{\epsilon}$ \cite[Lemma~5.1]{wainwright2019high}, the above inequality becomes
\begin{align} \label{fano}
    \inf\limits_{\hat{v}}\max\limits_{v_j\in \mathcal{P}(K\bigcap\bbs^{p-1}, \epsilon)} 
    \E\|\hat{v} - v_j\| 
    & \gtrsim
    \frac{\epsilon}{2}
    (1-\frac{n (\nu\wedge \nu^2)}{2\log N_{\epsilon}}
    -\frac{\log 2}{\log N_{\epsilon}})
\end{align}

\item 
Let $\epsilon^* = \argmax\limits_{\epsilon\geq0, \,\, N_{\epsilon}\geq 4} \epsilon \sqrt{\log N_{\epsilon}}$. By Lemma \ref{sudakov_vari}, when $w(K\bigcap\bbs^{p-1})\geq 64\sqrt{\log3}$, there exists $C_1 > 0$ such that
\begin{align}
\label{lower_eps_star}
    \epsilon^* \geq \frac{w(K\bigcap\bbs^{p-1})}{C_1\log{p}\sqrt{\log N_{\epsilon^*}}}
\end{align}
Plug in $\epsilon=\epsilon^*$ into (\ref{fano}). Notice that $N_{\epsilon^*}\geq 4$ so $\frac{\log 2}{\log N_{\epsilon^*}}\leq\frac{1}{2}$, we get
\begin{align*}
    \inf\limits_{\hat{v}}\max\limits_{v_j\in \mathcal{P}(K\bigcap\bbs^{p-1}, \epsilon^*)}
    \E\|\hat{v} - v_j\|
    &\gtrsim \frac{\epsilon^*}{2}
    (\frac{1}{2}-\frac{n(\nu\wedge \nu^2)}{2\log N_{\epsilon^*}})
\end{align*}

Pick $n(\nu\wedge \nu^2) = \log N_{\epsilon^*}/2$. Then we have
\begin{align*}
    \frac{n(\nu\wedge \nu^2)}{2\log N_{\epsilon^*}} = \frac{1}{4},
\end{align*}
and the lower bound becomes
\begin{align*}
    \inf\limits_{\hat{v}}\max\limits_{v_j\in \mathcal{P}(K\bigcap\bbs^{p-1},\epsilon^*)}
    \E\|\hat{v} - v_j\| \gtrsim \frac{\epsilon^*}{8}
\end{align*}
Plug in \eqref{lower_eps_star} to get
\begin{align*}
    \inf\limits_{\hat{v}}\max\limits_{v_j\in \mathcal{P}(K\bigcap\bbs^{p-1},\epsilon^*)}
    \E\|\hat{v} - v_j\| \gtrsim \frac{1}{8C_1}\frac{w(K\bigcap\bbs^{p-1})}{\log p \sqrt{\log N_{\epsilon^*}}}
\end{align*}
By $\log N_{\epsilon^*} = 2n(\nu\wedge \nu^2)$, the lower bound becomes
\begin{align*}
    \inf\limits_{\hat{v}}\max\limits_{v_j\in \mathcal{P}(K\bigcap\bbs^{p-1},\epsilon^*)}
    \E\|\hat{v} - v_j\| \gtrsim \frac{1}{16C_1}\frac{w(K\bigcap\bbs^{p-1})}{\log p \sqrt{n(\nu\wedge \nu^2)}}
\end{align*}
and we finally get 
$$\inf\limits_{\hat v} \sup\limits_{\bx \in K \bigcap \mathbb{S}^{p-1}} \mathbb{E} \|\hat v - \bx\| \geq \inf\limits_{\hat{v}}\max\limits_{v_j\in \mathcal{P}(K\bigcap\bbs^{p-1}, \epsilon^*)} 
\E\|\hat{v} - v_j\| 
\gtrsim
\frac{w(K\bigcap\bbs^{p-1})}{\log p(\nu\wedge\sqrt{\nu})\sqrt{n}}$$


Finally we show the last implication of the theorem. By Sudakov's inequality \cite[Theorem~8.1.13]{vershynin2018high}, $\log N_{\epsilon^*} \lesssim (\frac{w(K\bigcap\bbs^{p-1})}{\epsilon^*})^2$, and in the proof of Lemma \ref{sudakov_vari} we showed $\epsilon^* \geq \frac{w(K\bigcap\bbs^{p-1})}{4\sqrt{p}}$. 
Thus $\log N_{\epsilon^*}\lesssim 16 p$, and the theorem follows. 
\end{enumerate}

\subsection{Proof of Corollary \ref{minimax_coro}}
Based on the proof of Theorem \ref{minimax}, if we find a set $K'\subseteq K$ such that all vector pairs in $K'$ have a positive dot product, the packing set we choose would become $\mathcal{P}(K'\cap\mathbb{S}^{p-1}, \epsilon)$. The consequent proofs are the same as that of Theorem \ref{minimax}. Since all vector pairs in $K'$ have positive dot product, finally we get
$$\inf\limits_{\hat v} \sup\limits_{\bx \in K \bigcap \mathbb{S}^{p-1}} \mathbb{E} \|\hat v - \bx\|\wedge\|\hat v - \bx\| = 
\inf\limits_{\hat v} \sup\limits_{\bx \in K \bigcap \mathbb{S}^{p-1}} \mathbb{E} \|\hat v - \bx\|
\gtrsim
\frac{w(K'\bigcap\bbs^{p-1})}{\log p(\nu\wedge\sqrt{\nu})\sqrt{n}}$$

\section{Proof of Lemma \ref{gw_nonneg:lemma}}
\begin{proof}
By the definition $w(K^+\bigcap\mathbb{S}^{p-1})=\mathbb{E}\sup_{t\in K^+\bigcap\mathbb{S}^{p-1}}\langle g, t\rangle$, where $g \sim \mathcal{N}(0,I_p)$. This is equivalent with 
\begin{align*}
    \mathbb{E}\sup_{t\in\mathbb{S}^{p-1}}\langle g', t\rangle
\end{align*}
where the distribution of each $g_i'$ is a mixture of a folded standard normal and constant zero $\mathbb{P}(g_i'=|Z|) = \frac{1}{2}\quad\text{and}\quad\mathbb{P}(g_i'=0) = \frac{1}{2}$. For a standard normal random variable $Z$ the following tail bound holds
\begin{align*}
    \mathbb{P}[Z>t] & = \frac{1}{\sqrt{2\pi}}\int_t^{+\infty}e^{-\frac{x^2}{2}}dx\\
    & = \frac{1}{\sqrt{2\pi}}\int_0^{\infty}e^{-\frac{(t+y)^2}{2}}dy\\
    & \leq \frac{1}{\sqrt{2\pi}}e^{-\frac{t^2}{2}}\int_0^{\infty}e^{-ty}dy = \frac{1}{\sqrt{2\pi}}\frac{1}{t}e^{-\frac{t^2}{2}}
\end{align*}
When $t\geq1$, the above bound reduce to $\mathbb{P}[Z>t]\leq e^{-\frac{t^2}{2}}$; when $t<1$, we always have $\mathbb{P}[Z>t]\leq2e^{-\frac{t^2}{2}}$. Thus $\mathbb{P}[Z>t]\leq 2e^{-\frac{t^2}{2}}$ for $t>0$. Since $\mathbb{P}[|g_i|\geq t]=\mathbb{P}[g_i\geq t]=\frac{1}{2}\mathbb{P}[|Z|>t]=\mathbb{P}[Z>t]$, we have
\begin{align*}
    \mathbb{P}[|g_i|\geq t] \leq 2\exp(-t^2/2)
\end{align*}
By the definition of sub-Gaussian in terms of Orlicz norm \citep[(2.14)]{vershynin2018high}, the above inequality implies that $g_i'$ is sub-Gaussian with $\|g_i'\|_{\psi_2}=C$, where $C$ is an absolute constant. And since $\mathbb{E} (g_i')^2=\frac{1}{2}$, we have $\sqrt{2}\|g'\|-\sqrt{p}$ is also a sub-Gaussian random variable satisfying \citep[Theorem 3.1.1]{vershynin2018high}:
\begin{align*}
    \mathbb{P}\Big[\big|\|g'\|-\sqrt{\frac{p}{2}}\big|\geq t\Big] \leq 2\exp(-C_2t^2/\|g_i'\|_{\psi_2}^4) = 2\exp(-C_2t^2/C^4)
\end{align*}
where $C_2$ is an absolute constant. A simple integral shows that
\begin{align*}
    \mathbb{E}\Big|\|g'\|-\sqrt{\frac{p}{2}}\Big| & \leq \int_0^{\infty}\mathbb{P}\Big[\big|\|g'\|-\sqrt{\frac{p}{2}}\big|\geq t\Big]dt\\
    & \leq \int_0^{\infty}2\exp(-C_2t^2/C^4)dt\\
    & = C^2\sqrt{\pi/C_2}
\end{align*}
so that
\begin{align*}
    \mathbb{E}\|g'\|\asymp\sqrt{p}
\end{align*}
And then
\begin{align*}
    w(K^+\bigcap\mathbb{S}^{p-1}) = \mathbb{E}\sup_{t\in\mathbb{S}^{p-1}}\langle g', t\rangle = \mathbb{E}\|g'\|_2\asymp\sqrt{p}
\end{align*}
\end{proof}

\section{Proof of Proposition \ref{example_lower}}

    Now we construct a special example of monotone cone to calculate the lower bound of $L_2$ error rate. This example $K_2$ consists of vectors having three constant pieces. And the pairwise distance between vectors in $K_2$ is always greater than $\sqrt{2}\epsilon$.
\begin{example}
\label{k2}
Let $K_2 = \{a_i\}_{i=0}^{p-2}$ be a subset of $\mathbb{R}^p$ where
\begin{align*}
    a_i = \Big(
    \underbrace{0,..,0}_i, 
    \underbrace{\frac{\epsilon}{\sqrt{p-1-i}}, ..., \frac{\epsilon}{\sqrt{p-1-i}}}_{p-1-i}, \sqrt{1-\epsilon^2}
    \Big)
\end{align*}
\end{example}
Before investigating the lower bound, we need to introduce a result which tells the order of metric entropy of $K_2$. 

\begin{lemma}
\label{k2_packing}
For a monotone cone $K_2$ as defined in Example \ref{k2}, the cardinality of the maximum $\frac{\epsilon}{2}$-packing set is of order $\mathcal{O}\big(\log p\big)$.
\end{lemma}

\begin{proof}[Proof of Lemma \ref{k2_packing}] 

With out loss of generality, let $i' > i$. 
\begin{align*}
    ||a_i-a_{i'}||^2 & = 2 - 2a_i^Ta_{i'}\\
    & = 2 - 2\big[ 
    \epsilon^2\frac{\sqrt{p-1-i'}}{\sqrt{p-1-i}}
    +1-\epsilon^2
    \big] \\
    & = 2\epsilon^2\big[
    \frac{p-1-i-\sqrt{p-1-i'}\sqrt{p-1-i}}{p-1-i}
    \big]\\
    & \geq \frac{2\epsilon^2(i'-i)}{p-1-i}
\end{align*}
To let $||a_i-a_{i'}||\geq \frac{\epsilon}{2}$, we need $\frac{i'-i}{p-1-i} \geq \frac{1}{8}
\quad \Rightarrow \quad 
i'\geq \frac{1}{8}(p-1)+\frac{7}{8}i$.
\vspace{9pt} \quad \\
Thus the max $\frac{\epsilon}{2}$-packing $\{a_{c_k}\}_{k=0}^n$ can be constructed by a sequence $S_c = \{c_k\}_{k=0}^n$ where 
\begin{align*}
    c_0 & = 0 \\
    c_{k+1} &=
    \lceil 
    \frac{1}{8}(p-1)+\frac{7}{8}c_k
    \rceil \\
    c_n & \leq p - 2 \\
    \frac{(p-2)-c_n}{p-1-c_n} < \frac{1}{8}
    \quad \Rightarrow \quad 
    c_n & \geq p - \frac{15}{7} 
\end{align*}

\begin{itemize}
    \item \textbf{Lower bound of $|S_c|$:} \\
    In order to get a lower bound of $|S_c|$, construct another sequence $S_b = \{b_k\}_{k=0}^m$ such that 
\begin{align*}
    b_0 & = 0 \\
    b_{k+1} &=
    \frac{1}{8}(p-1)+\frac{7}{8}b_k + 1 \\
    b_m & \geq p - \frac{15}{7} 
\end{align*}
It is easy to see $|S_b| = |S_c|$ since $c_k \leq b_k$ for all $k$. \vspace{4pt}\\
Furthermore one can get $b_{k+2}-b_{k+1} = \frac{7}{8}(b_{k+1}-b_k)$ from the above equations. Then
\begin{align*}
    b_m & = (b_m - b_{m-1}) + 
    ... + (b_1 - b_0) \\
    & = (b_1-b_0)(\frac{7}{8})^{m-1} + 
    ... + (b_1-b_0)(\frac{7}{8}) + 
    (b_1-b_0) \\
    & = (\frac{1}{8}p+\frac{7}{8})\big[ 
    (\frac{7}{8})^0 + 
    (\frac{7}{8})^1+
    ...+(\frac{7}{8})^{m-1}
    \big] \\
    & = (p+7)[1-(\frac{7}{8})^{m}]
\end{align*}

Notice that $b_m \geq p-\frac{15}{7}$ is required. Thus 
\begin{align*}
    b_m = (p+7)[1-(\frac{7}{8})^m] \geq p - \frac{15}{7}
    \qquad \Rightarrow \qquad 
    m \geq \frac{\log(p+7)}{\log \frac{8}{7}} -
    \frac{\log(\frac{15}{7} + 7)}{\log\frac{8}{7}}
\end{align*}
Therefore $|S_c| = |S_b| = \mathcal{O}(\log p)$.

\item \textbf{Upper bound of $|S_c|$:} \\
In order to get an upper bound of $|S_c|$, construct another sequence $S_d = \{d_k\}_{k=0}^l$ such that 
\begin{align*}
    d_0 & = 0 \\
    d_{k+1} &=
    \frac{1}{8}(p-1)+\frac{7}{8}d_k \\
    d_l & \leq p - 2 
\end{align*}
It is easy to see $|S_d| = |S_c|$ since $c_k \geq d_k$ for all $k$. \vspace{4pt}\\
Similar as the proof of lower bound of $|S_c|$, from $d_{k+1} =
\frac{1}{8}(p-1)+\frac{7}{8}d_k$ we derive the expression of $d_l$ as 
\begin{align*}
    d_l = (p+7)[1-(\frac{7}{8})^{l}]
\end{align*}
By $d_l\leq p - 2$ we have 
\begin{align*}
    d_l = (p+7)[1-(\frac{7}{8})^{l}] \leq p - 2
    \qquad \Rightarrow \qquad 
    l \leq \frac{\log(p+7)}{\log\frac{8}{7}} + 
    \frac{\log9}{\log\frac{8}{7}}
\end{align*}
Therefore $|S_c| = |S_d| = \mathcal{O}(\log p)$.
\end{itemize}
\end{proof}

\subsection{Proof of Proposition \ref{example_lower}}

Let $\mathcal{P}(K_2,\frac{\epsilon}{2})$ be the max $\frac{\epsilon}{2}$-packing of $K_2$, and $M_{\frac{\epsilon}{2}} = |\mathcal{P}(K_2,\frac{\epsilon}{2})|$. Define a distance function $d(v,v')=\|v - v'\|\wedge\|v + v'\|$. For any $v, v'\in \mathcal{P}(K_2,\frac{\epsilon}{2})$ with $v\neq v'$, we have $v\T v'\geq 0$. Then
\begin{align*}
    d(v,v') = \|v - v'\| \geq \frac{\epsilon}{2}
\end{align*}
And by the nature of the construction of $K_2$ in Example \ref{k2}
\begin{align*}
    d(v,v') = \|v - v'\| = \sqrt{2 - 2v\T v'} = 
    \sqrt{ 2 - 2\big[ 
    \epsilon^2\frac{\sqrt{p-1-i'}}{\sqrt{p-1-i}}
    +1-\epsilon^2
    \big] }
    = \sqrt{ 2\epsilon^2\big[
    1 - \frac{\sqrt{p-1-i'}}{\sqrt{p-1-i}}
    \big] }
    \leq \sqrt{2}\epsilon
\end{align*}

By Lemma \ref{fano_yu}, the lower bound of minimax risk is derived by Fano's method
\begin{align*}
    \inf\limits_{\hat{v}}\max\limits_{v_j\in \mathcal{P}(K_2,\frac{\epsilon}{2}) } 
    \E[\|\hat{v} - v_j\|\wedge\|\hat{v} + v_j\|] 
    &\gtrsim 
    \frac{\epsilon}{4}(1-\frac{ nD+\log2}{\log M_{\frac{\epsilon}{2}}}) 
\end{align*}
The Kullback-Leibler divergence of $P_v, P_v'$ can be upper bounded as
\begin{align*}
    D = \frac{\nu^2}{2(\nu+1)}[1-\langle v, v'\rangle^2]
    \lesssim \frac{(\nu\wedge\nu^2)\epsilon^2}{2}
\end{align*}
Thus the minimax lower bound becomes
\begin{align*}
    \inf\limits_{\hat{v}}\max\limits_{v_j\in \mathcal{P}(K_2,\frac{\epsilon}{2}) } 
    \E[\|\hat{v} - v_j\|\wedge\|\hat{v} + v_j\|] 
    & \gtrsim
    \frac{\epsilon}{4}
    (1 - \frac{n(\nu\wedge\nu^2)\epsilon^2}{2\log M_{\frac{\epsilon}{2}}} - \frac{\log2}{\log M_{\frac{\epsilon}{2}}})
\end{align*}

By Lemma \ref{k2_packing}, $\log N_{\frac{\epsilon}{2}} \sim \mathcal{O}\big(\log p\big)$, then 
\begin{align*}
    \inf\limits_{\hat{v}}\max\limits_{v_j\in \mathcal{P}(K_2,\frac{\epsilon}{2}) } 
    \E[\|\hat{v} - v_j\|\wedge\|\hat{v} + v_j\|]
    & \gtrsim 
    \frac{\epsilon}{4}
    (1 - \frac{n(\nu\wedge\nu^2)\epsilon^2}{2\log\log p} - \frac{\log2}{\log\log p}) \\
    & \gtrsim
    \frac{\epsilon}{4}
    (\frac{1}{2} - \frac{n(\nu\wedge\nu^2)\epsilon^2}{2\log\log p})
    \qquad \qquad 
    \text{if } p \geq e^4
\end{align*}
Pick $\epsilon = \sqrt{\frac{\log\log p}{2n(\nu\wedge\nu^2)}}$, and plug in to the above inequality get 
\begin{align*}
    \inf\limits_{\hat{v}}\max\limits_{v_j\in \mathcal{P}(K_2,\frac{\epsilon}{2}) } 
    \E[\|\hat{v} - v_j\|\wedge\|\hat{v} + v_j\|] 
    &\gtrsim \epsilon \gtrsim 
    \frac{\sqrt{\log\log p}}{(\nu\wedge\sqrt{\nu})\sqrt{n}}
\end{align*}
Let $M$ be the monotone cone in $\mathbb{R}^p$. Finally
\begin{align*}
    \inf\limits_{\hat{v}}\max\limits_{v_j\in M }
    \E[\|\hat{v} - v_j\|\wedge\|\hat{v} + v_j\|] 
    \geq
    \inf\limits_{\hat{v}}\max\limits_{v_j\in \mathcal{P}(K_2,\frac{\epsilon}{2}) } 
    \E[\|\hat{v} - v_j\|\wedge\|\hat{v} + v_j\|] 
    \gtrsim
    \frac{\sqrt{\log\log p}}{(\nu\wedge\sqrt{\nu})\sqrt{n}}
\end{align*}

\pagebreak

\bibliographystyle{apalike}
\bibliography{bibfile}

\end{document}